 \documentclass[12pt, reqno]{amsart}
\usepackage[margin=1.0in]{geometry}
\usepackage{amsfonts}

\usepackage{setspace}
\usepackage{graphicx}

\usepackage{hyperref}
\usepackage{graphicx}
\usepackage{amsmath, amsthm, amscd, amsfonts, amssymb, graphicx, color}
 \usepackage{url}
\usepackage{enumerate}
 \makeatletter
\let\reftagform@=\tagform@
\def\tagform@#1{\maketag@@@{(\ignorespaces\textcolor{blue}{#1}\unskip\@@italiccorr)}}
\renewcommand{\eqref}[1]{\textup{\reftagform@{\ref{#1}}}}
\makeatother
\usepackage{hyperref}
\hypersetup{colorlinks=true, linkcolor=red, anchorcolor=green,
citecolor=cyan, urlcolor=red, filecolor=magenta, pdftoolbar=true}

\usepackage{epsfig}        % For postscript
\usepackage{epic,eepic}       % For epic and eepic output from xfig

\newtheorem{theorem}{Theorem}
\theoremstyle{plain}

\newtheorem{corollary}{Corollary}

\newtheorem{lemma}{Lemma}

\newtheorem{remark}{Remark}

\numberwithin{equation}{section}

 \DeclareMathOperator{\osc}{Osc}

\begin{document}

%\title[Two-point Ostrowski and Ostrowski-Gr\"{u}ss type inequalities]{Two-point Ostrowski and Ostrowski-Gr\"{u}ss type inequalities with applications}

\title[Two--point Ostrowski inequalities]{ Two--point Ostrowski  and Ostrowski--Gr\"{u}ss type inequalities with applcations}

\author[M.W. Alomari]{Mohammad W. Alomari}

\address{Department of Mathematics, Faculty of Science and
Information Technology, Irbid National University, P.O. Box 2600,
Irbid, P.C. 21110, Jordan.} \email{mwomath@gmail.com}

\date{\today}
\subjclass[2010]{Primary: 41A10, 41A44, 41A58, 41A80.  Secondary:
26D15, 65D30, 65D32}

\keywords{ Approximations, Expansions, Quadrature rule, Euler--Maclaurin formula, Ostrowski inequality}

\begin{abstract}
In this work, an extension of two-point Ostrowski's formula for
$n$-times differentiable functions is proved. A generalization of
Taylor formula is deduced. An identity of Fink type for this
extension is provided. Error estimates for the considered formulas
are also given. Two--point  Ostrowski--Gr\"{u}ss type inequalities are
pointed out. An expansion of Guessab--Schmeisser two points
formula for $n$-times differentiable functions via  Fink type
identity is established. Generalization of the main result for
harmonic sequence of polynomials is established. Several bounds of
the presented results are proved. 
\end{abstract}

\maketitle

%=============================================================================
\section{Introduction}
%=============================================================================

To approximate $\int_a^b{f\left(t\right)dt}$, by a general
one-point rule
\begin{align}
\int_a^b{f\left(t\right)dt} =
\left(b-a\right)f\left(x\right)+E\left(f,x\right), \qquad \forall
x\in \left[a,b\right],\label{eq1.1}
\end{align}
Let us suppose  $f$ is differentiable on $\left[a,b\right]$ and
$f^{\prime }\in L[a,b]$. If $\left\| {f^{\prime
}}\right\|_{\infty}=\mathop {\sup }\limits_{x \in \left[ {a,b}
\right]} \left| {f^{\prime }(x)}\right| \leq \infty$. Therefore,
the Ostrowski estimates \cite{O1} reads:
\begin{align}
\label{eq1.2}\left\vert {E\left(f,x\right)}\right\vert \leq\left[
{\frac{\left( {b - a} \right)^2 }{4} +
\left({x-\frac{a+b}{2}}\right)^2} \right] \left\| {f^{\prime}}
\right\|_{\infty},
\end{align}
for all $x \in [a,b]$. The constant $\frac{1}{4}$ is the best
possible.

In 1976  Milovanovi\'{c} and Pe\v{c}ari\'{c} \cite{Milovanovic}
provided their famous generalization of \eqref{eq1.2}  via Taylor
series, where they proved that:
\begin{align}
\int_a^b{f\left(t\right)dt} = \frac{b-a}{n}\left( {f\left( x
\right) + \sum\limits_{k = 1}^{n - 1} {F_k \left( x \right)} }
\right)+E_n\left(f,x\right), \qquad \forall x\in
\left[a,b\right],\label{eq1.3}
\end{align}
such that
\begin{align}
F_k \left( x \right) = \frac{{n - k}}{{n!}}\cdot\frac{{f^{\left(
{k - 1} \right)} \left( a \right)\left( {x - a} \right)^k  -
f^{\left( {k - 1} \right)} \left( b \right)\left( {x - b}
\right)^k }}{{b - a}}.\label{eq1.4}
\end{align}
with error estimates
\begin{align}
\label{eq1.5}\left| { E_n\left(f,x\right)} \right| \le
\frac{{\left( {x - a} \right)^{n + 1}  + \left( {b - x} \right)^{n
+ 1} }}{{ n\left( {n + 1} \right)!}} \cdot\left\| {f^{\left( n
\right)} } \right\|_\infty,
\end{align}

In 1992, Fink studied \eqref{eq1.4} in different presentation, he
introduced a new representation of  real $n$-times differentiable
function whose $n$-th derivative $(n\ge1)$ is absolutely
continuous by combining Taylor series and Peano kernel approach
together. Namely, in \cite{Fink} we find:
\begin{align}
E_n\left(f,x\right) =  \frac{1}{n! }\int_a^b {\left( {x - t}
\right)^{n - 1} p\left( {t,x} \right)f^{\left( n \right)} \left( t
\right)dt},\label{eq1.6}
\end{align}
for all $x\in\left[{a,b}\right]$, where
\begin{align}
p\left( {t,x} \right) = \left\{ \begin{array}{l}
 t - a,\,\,\,\,\,\,\,\,\,\,\,\,\,\,\,\,\,\, t \in \left[ {a,x} \right] \\
 t - b,\,\,\,\,\,\,\,\,\,\,\,\,\,\,\,\,\,\, t \in \left[ {x,b} \right] \\
 \end{array} \right..\label{eq1.7}
\end{align}
In the same work, Fink proved the following bound of
\eqref{eq1.6}.
\begin{align}
\label{eq1.8}\left| {E_n\left(f,x\right)} \right| \le C\left(
{n,p,x} \right)\left\| {f^{\left( n \right)} } \right\|_p,
\end{align}
where $\left\|  \cdot  \right\|_r$, $1 \le r \le \infty$ are the
usual Lebesgue norms on $L_r [a, b]$, i.e.,
\begin{align*}
\left\| f \right\|_\infty : = ess\mathop {\sup }\limits_{t \in
\left[ {a,b} \right]} \left| {f\left( t \right)} \right|,
\end{align*}
 and
\begin{align*}
\left\| f \right\|_r : = \left( {\int_a^b {\left| {f\left( t
\right)} \right|^r dt} } \right)^{1/r},\,\,\,1 \le r < \infty,
\end{align*}
such that
\begin{align*}
C\left( {n,p,x} \right) = \frac{1}{ n!} {\rm{B}}^{1/q} \left(
{\left( {n - 1} \right)q + 1,q + 1} \right)\left[ {\left( {x - a}
\right)^{nq + 1}  + \left( {b - x} \right)^{nq + 1} }
\right]^{1/q},
\end{align*}
for $1 < p \le \infty$, $ {\rm{B}}\left(\cdot,\cdot\right)$ is the
beta function, and for $p=1$
\begin{align*}
C\left( {n,1,x} \right) = \frac{{\left( {n - 1} \right)^{n - 1}
}}{n^n n!}\max \left\{ {\left( {x - a} \right)^n ,\left( {b - x}
\right)^n } \right\}.
\end{align*}
 All previous
bounds are sharp.

Indeed Fink representation can be considered as the first elegant
work (after Darboux work \cite{Kythe}, p.49) that combines two
different approaches together, so that Fink representation is not
less important than Taylor expansion itself. So that, many authors
were interested to study Fink representation approach, more
detailed and related results can be found in
\cite{Aljinovic1},\cite{Aljinovic2},\cite{G1},\cite{G2} and
\cite{Dedic1}.

In 2002 and the subsequent years after that, the Ostrowski's
inequality entered in a new phase of modifications and
developments. A new inequality of Ostrowski's type was born, where
Guessab and Schmeisser in \cite{Guessab} discussed an inequality
from algebraic and analytic points of view which is in connection
with Ostrowski inequality; called  `\emph{the companion of
    Ostrowski's inequality}' as suggested later by Dragomir in
\cite{Dragomir1}. The main part of Guessab--Schmeisser inequality
reads the difference between symmetric values of a real function
$f$ defined on $[a,b]$ and its weighed value, i.e.,
\begin{align*}
\frac{{f\left( x \right) + f\left( {a + b - x} \right)}}{2} -
\frac{1}{{b - a}}\int_a^b {f\left( t \right)dt}, \qquad x \in
\left[{a,\frac{a+b}{2}}\right].
\end{align*}
Namely, in the significant work \cite{Guessab} we find the first
primary result is that:
\begin{theorem}
    \label{thm2} Let $f : [a,b] \to \mathbb{R}$ be satisfies the
    H\"{o}lder condition of order $r\in (0,1]$. Then for each $x \in
    [a,\frac{a + b}{2}]$, the we have the inequality
    \begin{align}
    \label{eq1.9} \left| {\frac{{f\left( x \right) + f\left( {a + b -
                    x} \right)}}{2} - \frac{1}{{b - a}}\int_a^b {f\left( t \right)dt}} \right|
    \le \frac{M}{{b - a}}\frac{{\left( {2x - 2a} \right)^{r + 1}  +
            \left( {a + b - 2x} \right)^{r + 1} }}{{2^r \left( {r + 1}
            \right)}}.
    \end{align}
    This inequality is sharp for each admissible $x$. Equality is
    attained if and only if $f=\pm M f_{*}+c$ with $c\in\mathbb{R}$
    and
    \begin{align*}
    f_* \left( t \right) = \left\{ \begin{array}{l}
    \left( {x - t} \right)^r ,\,\,\,\,\,\,\,\,\,\,\,\,\,\,\,\,\,\,\,\,\,\,\,{\rm{if}}\,\,\,a \le t \le x \\
    \left( {t - x} \right)^r ,\,\,\,\,\,\,\,\,\,\,\,\,\,\,\,\,\,\,\,\,\,\,\,{\rm{if}}\,\,\,x \le t \le \frac{{a + b}}{2} \\
    f_* \left( {a + b - t} \right),\,\,\,\,\,\,\,\,\,{\rm{if}}\,\,\,\frac{{a + b}}{2} \le t \le b \\
    \end{array} \right..
    \end{align*}
\end{theorem}

In the same work \cite{Guessab}, the authors discussed and
investigated  \eqref{eq1.8} for other type of assumptions. Among
others, a brilliant representation (or identity) of $n$-times
differentiable functions whose $n$-th derivatives are piecewise
continuous was established as follows:
\begin{theorem}
    \label{thm3}Let $f$ be a function defined on $[a, b]$ and having
    there a piecewise continuous $n$-th derivative. Let $Q_n$ be any
    monic polynomial of degree $n$ such that $Q_n\left(t\right)=
    \left(-1\right)^n Q_n\left(a+b-t\right)$. Define
    \begin{align*}
    K_n \left( t
    \right) = \left\{ \begin{array}{l}
    \left( {t - a} \right)^n ,\,\,\,\,\,\,\,\,\,\,\,\,{\rm{if}}\,\,\,\,a \le t \le x \\
    \\
    Q_n \left( t \right),\,\,\,\,\,\,\,\,\,\,\,\,\,\,\,\,{\rm{if}}\,\,\,\,x \le t \le a + b - x \\
    \\
    \left( {t - b} \right)^n ,\,\,\,\,\,\,\,\,\,\,\,\,{\rm{if}}\,\,\,\,a + b - x \le t \le b \\
    \end{array} \right..
    \end{align*}
    Then,
    \begin{align}
    \label{eq1.10}\int_a^b {f\left( t \right)dt}  = \left( {b - a}
    \right)\frac{{f\left( x \right) + f\left( {a + b - x} \right)}}{2}
    + E\left( {f;x} \right)
    \end{align}
    where,
    \begin{multline*}
    E\left( {f;x} \right) = \sum\limits_{\nu  = 1}^{n - 1} {\left[
        {\frac{{\left( {x - a} \right)^{\nu  + 1} }}{{\left( {\nu  + 1}
                    \right)!}}-\frac{{Q_n^{\left( {n - \nu  - 1} \right)} \left( x
                    \right)}}{{n!}}} \right]\left[ {f^{\left( \nu  \right)} \left( {a
                + b - x} \right) + \left( { - 1} \right)^\nu  f\left( x \right)}
        \right]}
    \\
    + \frac{{\left( { - 1} \right)}}{{n!}}\int_a^b {K_n \left( t
        \right)f^{\left( n \right)} \left( t \right)dt}.
    \end{multline*}
\end{theorem}
This generalization \eqref{eq1.10} can be considered  as a
companion type expansion of Euler--Maclaurin formula that expand
symmetric values of real functions. In this way, families of
various quadrature rules can be presented, as shown -for example-
in \cite{Kovac}. Therefore, since 2002 and after the presentation
of \eqref{eq1.9}, several authors have studied, developed  and
established  new presentations concerning \eqref{eq1.10} using
several approaches and different tools, for this purpose see the
recent survey \cite{Dragomir}.

Far away from this, in the last thirty years  the concept of
harmonic sequence of polynomials or Appell polynomials have been
used at large in numerical integrations and expansions theory of
real functions. Let us recall that, a  sequence of polynomials
$\left\{{P_k\left(t, \cdot\right)}\right\}_{k=0}^{\infty}$
satisfies the Appell condition (see \cite{Appell}) if
$\frac{\partial}{\partial t} P_k \left( {t, \cdot } \right)=
P_{k-1} \left( {t, \cdot } \right)$ $(\forall k\ge1)$ with $P_0
\left( {t, \cdot } \right)=1$, for all well-defined order pair
$\left( {t, \cdot } \right)$. A slightly different definition was
considered in \cite{matic1}.

In 2003, motivated by work of Mati\'{c} et. al. \cite{matic1},
Dedi\'{c} et. al. in \cite{Dedic1}, introduced the following smart
generalization of Ostrowski's inequality via harmonic sequence of
polynomials:
\begin{multline}
\label{eq1.11}\frac{1}{n}\left[ {f\left( x \right) +
    \sum\limits_{k = 1}^{n - 1} {\left( { - 1} \right)^k P_k \left( x
        \right)f^{\left( k \right)} \left( x \right)}  + \sum\limits_{k =
        1}^{n - 1} {\widetilde{F_k }\left( {a,b} \right)} } \right]
\\
= \frac{{\left( { - 1} \right)^{n - 1} }}{{\left( {b - a}
        \right)n}}\int_a^b {P_{n - 1} \left( t \right)p\left( {t,x}
    \right)f^{\left( n \right)} \left( t \right)dt},
\end{multline}
where $P_k$ is a harmonic sequence of polynomials satisfies that
$P^{\prime}_k=P_{k-1}$ with $P_0=1$,
\begin{align}
\label{eq1.12}\widetilde{F_k }\left( {a,b} \right) = \frac{{\left(
        { - 1} \right)^k \left( {n - k} \right)}}{{b - a}}\left[ {P_k
    \left( a \right)f^{\left( {k - 1} \right)} \left( a \right) - P_k
    \left( b \right)f^{\left( {k - 1} \right)} \left( b \right)}
\right]
\end{align}
and $p\left( {t,x} \right)$ is given in \eqref{eq1.5}. In
particular, if we take $P_k\left( {t}
\right)=\frac{\left(t-x\right)^k}{k!}$ then we refer to Fink
representation \eqref{eq1.5}.

In 2005, Dragomir \cite{Dragomir1}  proved the following bounds of
the companion of Ostrowski's inequality for absolutely continuous
functions.
\begin{theorem}\label{thm3}
    Let $f:I\subset \mathbb{R}\rightarrow \mathbb{R}$ be an absolutely
    continuous function on $[a,b]$. Then we have the inequalities
    \begin{multline}
    \label{eq1.13}\left| {\frac{{f\left( x \right) + f\left( {a + b -
                    x} \right)}}{2} - \frac{1}{{b - a}}\int_a^b {f\left( t \right)dt}
    } \right| \\\le \left\{ \begin{array}{l}
    \left[ {\frac{1}{8} + 2\left( {\frac{{x - {\textstyle{{3a + b} \over 4}}}}{{b - a}}} \right)^2 } \right]\left( {b - a} \right)\left\| {f'} \right\|_\infty  ,\,\,\,\,\,\,\,\,\,f' \in L_\infty  \left[ {a,b} \right] \\
    \\
    \frac{{2^{1/q} }}{{\left( {q + 1} \right)^{1/q} }}\left[ {\left( {\frac{{x - a}}{{b - a}}} \right)^{q + 1}  - \left( {\frac{{{\textstyle{{a + b} \over 2}} - x}}{{b - a}}} \right)^{q + 1} } \right]^{1/q} \left( {b - a} \right)^{1/q} \left\| {f'} \right\|_{\left[ {a,b} \right],p} , \\
    \,\,\,\,\,\,\,\,\,\,\,\,\,\,\,\,\,\,\,\,\,\,\,\,\,\,\,\,\,\,\,\,\,\,\,\,\,\,\,\,\,\,\,\,\,\,\,\,\,\,\,\,\,\,\,\,\,\,\,\,\,\,\,\,\,\,\,\,\,\,\,\,\,\,p > 1,\frac{1}{p} + \frac{1}{q} = 1,\,and\,f' \in L_p \left[ {a,b} \right] \\
    \left[ {\frac{1}{4} + \left| {\frac{{x - {\textstyle{{3a + b} \over 4}}}}{{b - a}}} \right|} \right]\left\| {f'} \right\|_{\left[ {a,b} \right],1}  \\
    \end{array} \right.
    \end{multline}
    for all $x \in [a,\frac{a + b}{2}]$. The constants $\frac{1}{8}$
    and  $\frac{1}{4}$ are the best possible in (\ref{eq1.13}) in the
    sense that it cannot be replaced by smaller constants.

\end{theorem}

In the last fifteen years, constructions of quadrature rules using
expansion of an arbitrary function in Bernoulli polynomials and
Euler--Maclaurin's type formulae have been established, improved
and investigated. These approaches permit many researchers to work
effectively in the area of numerical integration where several
error approximations of various quadrature rules presented with
high degree of exactness. Mainly,  works of Dedi\'{c} et al.
\cite{Dedic1}--\cite{Dedic2}, Aljinovi\'{c} et al.
\cite{Aljinovic1}, \cite{Aljinovic2}, Kova\'{c} et al.
\cite{Kovac} and others, received positive responses and good
interactions from other focused researchers. Among others,
Franji\'{c} et al. in several works (such as \cite{F1}--\cite{F5})
constructed several Newton-Cotes and Gauss quadrature type rules
using a certain expansion of real functions in Bernoulli
polynomials or Euler--Maclaurin's type formulae.

Unfortunately, the expansions \eqref{eq1.3}, \eqref{eq1.10} and
\eqref{eq1.11} have not been  used to construct quadrature rules
yet. It seems these expansions were abandoned or neglected in
literature because most of authors are still use the classical
Euler--Maclaurin's formula and expansions in Bernoulli
polynomials.

In order to generalize Guessab--Schmeisser formula \eqref{eq1.9}
for symmetric and non-symmetric points,  we have introduced the
following general quadrature rule in the recent work
\cite{alomari1}.
\begin{align}
\label{eq1.14}\int_a^b {f\left( s \right)ds}=Q\left( {f;y,x,z}
\right)+E\left( {f;y,x,z} \right),
\end{align}
where $Q\left( {f;y,x,z} \right)$ is the general two-point formula
\begin{align}
\label{eq1.15}Q\left( {f;y,x,z} \right):=\left( { x-a}
\right)f\left( y \right)+\left( { b- x} \right)f\left( z \right),
\end{align}
for all $x\in \left[y,z\right]$ with $a \le y \le z \le b$, with
error term
\begin{align}
\label{eq1.16}E\left( {f;y,x,z} \right):=\int_a^b {K\left( {s;y
,x,z} \right)df\left( s \right)},
 \end{align}
where
\begin{align*}
 K\left( {s;y ,x,z } \right) = \left\{ \begin{array}{l}
 s - a,\,\,\,\,\,\,\,\,\,\,a \le s \le y
 \\
 s - x ,\,\,\,\,\,\,\,\,\,\,y  < s < z
 \\
 s - b,\,\,\,\,\,\,\,\,\,\,z \le s \le b
 \end{array} \right.,
 \end{align*}
 for all $x \in \left[y,z\right]\subseteq \left[a,b\right]$.

In the same work \cite{alomari1}, we provided error estimates of
$E\left( {f;y,x,z} \right)$ involving functions possess at most  first
derivatives. Namely, we proved the following generalization of
Ostrowski inequality which is considered as two-point inequality.
\begin{align}
\label{eq1.17} \left\vert {E\left( {f;y,x,z} \right) }\right\vert
\le \left[{ \frac{\left( {x - a} \right)^2+\left( {b - x}
\right)^2}{4} + \left( {y  - \frac{{a + x}}{2}} \right)^2 + \left(
{z - \frac{{x + b}}{2}} \right)^2  } \right] \cdot \left\| {f'}
\right\|_{\infty ,\left[ {a,b} \right]},
\end{align}
for all $a \le y \le x \le  z \le b$. The constant $\frac{1}{4}$
is the best possible.

In particular, we deduced a sharp error estimates for the average
of general two point formula as follows:
\begin{align}
\label{eq1.18} \left\vert {E\left( {f;y,\frac{a+b}{2},z} \right)
}\right\vert \le \left[{ \frac{\left( {b - a} \right)^2}{8} +
\left( {y  - \frac{{3a + b}}{4}} \right)^2 + \left( {z - \frac{{a
+ 3b}}{4}} \right)^2  } \right] \cdot \left\| {f'}
\right\|_{\infty ,\left[ {a,b} \right]}.
\end{align}
The constant $\frac{1}{8}$ is the best possible.

\begin{remark}
If we choose $y=x=z$, the we recapture the Ostrowski inequality
\eqref{eq1.2}. By setting $z=a+b-y$ and choose $x=\frac{a+b}{2}$,
we get the Guessab--Schmeisser formula \eqref{eq1.9} for $n=1$.
Also, for $y=a$ and $z=b$ we recapture the generalized trapezoid
inequality that was introduced in \cite{Cerone}.
\end{remark}
The author of this paper has been given serious attention to
Guessab--Schmeisser inequality  in the works
\cite{alomari1}--\cite{alomari8}.   For other related results and
generalizations concerning Ostrowski's inequality and its
applications we refer the reader to
\cite{G1}--\cite{Cerone3},\cite{Cerone4}, \cite{Dedic2},
\cite{Dragomir1}--\cite{Dragomir3}, \cite{Liu1}, 
 and  \cite{Sofo}.

This work has several aims and goals, the first aim is to
generalize Guessab--Schmeisser formula for symmetric and
non-symmetric two points  for $n$-times differentiable functions
via Peano functional approach and Fink type identity and thus
provide several type of bounds for the remainder formula. An
extension of our result \eqref{eq1.14} for $n$-times
differentiable functions is proved. A Fink type identity for this
extension is provided. Error estimates for the considered formulas
are also given. The second goal, is to highlight the importance of
these expansions by giving a serious attention to their applicable
usefulness in constructing various quadrature rules. The third
aim,  is to spotlight the role of \v{C}eby\v{s}ev functional in
integral approximations.

This work is organized as follows: in the next section, a
generalization of Two point formula for $n$-times differentiable
is considered. A Taylor expansion would be then deduced as a
special case. Indeed, this Two point formula can be read as an
expansion of real analytic function near two point instead of one
point. Remainder estimates for the constructed formula are
provided. The constant in the obtained estimates are shown to be
the best possible. A mean value theorem is also given. Bounds for
functions possess H\"{o}lder continuity of order $r \in
\left(0,1\right]$ are proved. In Section \ref{sec3}, A Fink type
identity for the presented two point formula is established. For
instance, a Guessab--Schmeisser two points formula for $n$-times
differentiable functions via Fink type identity is also deduced.
Bounds for the remainder term of the presented formula are proved.
In Section \ref{sec4}, Inequalities of two-point
Ostrowski--Gr\"{u}ss type are introduced. Bounds for the remainder
of some previously obtained formulas  via Chebyshev-Gr\"{u}ss type
inequalities are presented. A practicle and
applicable example of the previous section; in fact, bounds for
the Guessab--Schmeisser two points formula are given.

%=============================================================================
\section{Two-Point Ostrowski's Inequality}\label{sec2}
%=============================================================================

\subsection{Expansions}

Let $I$ be a real interval and $a,b$ in $I^{\circ}$ (the interior
of $I$) with $a<b$.  Let $H_{\left(n,i\right)}$ be a harmonic
sequences of polynomials for $i=1,2,3$, and let $f:I \to
\mathbb{R}$ be such that $f^{(n-1)}$ $(n\ge1)$ is of bounded
variation on $I$ for all $n\ge1$.

For all $y,z \in \left[a,b\right]$ with $y\le z$. Define the
kernel $S_n:\left[a,b\right]^3 \to \mathbb{R}$ given by
\begin{align}
S_n\left( {t,y,z} \right) = \left\{ \begin{array}{l}
 \frac{{\left( {t - a}
\right)^n }}{{n!}},\,\,\,\,\,\,   t \in \left[ {a,y} \right] \\
\\
 \frac{{\left( {t - x}
\right)^n }}{{n!}},\,\,\,\,\,\, t \in \left( {y,z} \right) \\
\\
 \frac{{\left( {t - b}
\right)^n }}{{n!}},\,\,\,\,\,\,  t \in \left[ {z,b} \right] \\
 \end{array} \right..\label{eq2.1}
\end{align}
Integrating by parts
\begin{multline*}
\left( { - 1} \right)^n \int_a^y {\frac{{\left( {t - a} \right)^n
}}{{n!}}df^{\left( {n - 1} \right)} \left( t \right)}  = \left( {
- 1} \right)^n  \frac{{\left( {y - a} \right)^n }}{{n!}}f^{\left(
{n - 1} \right)} \left( y \right)
\\
+ \left( { - 1} \right)^{n - 1} \int_a^y {\frac{{\left( {t - a}
\right)^{n-1} }}{{\left(n-1\right)!}} f^{\left( {n - 1} \right)}
\left( t \right)dt}.
\end{multline*}
Successive integrations by parts yield
\begin{align}
\left( { - 1} \right)^n \int_a^y {\frac{{\left( {t - a} \right)^n
}}{{n!}}df^{\left( {n - 1} \right)} \left( t \right)} =
\sum\limits_{k = 1}^n {\left( { - 1} \right)^k \frac{{\left( {y -
a} \right)^k }}{{k!}}f^{\left( {k - 1} \right)} \left( y \right)}
+ \int_a^y {f\left( t \right)dt}.\label{eq2.2}
\end{align}
Similarly,
\begin{multline}
\left( { - 1} \right)^n \int_y^z {\frac{{\left( {t - x} \right)^n
}}{{n!}}df^{\left( {n - 1} \right)} \left( t \right)}
\\
= \sum\limits_{k = 1}^n {\left( { - 1} \right)^k \left[
{\frac{{\left( {z-x} \right)^k }}{{k!}}f^{\left( {k - 1} \right)}
\left( z \right) - \frac{{\left( {y-x} \right)^k}}{{k!}}f^{\left(
{k - 1} \right)} \left( y \right)} \right]}  + \int_y^z {f\left( t
\right)dt},\label{eq2.3}
\end{multline}
and
\begin{align}
\left( { - 1} \right)^n \int_z^b {\frac{{\left( {t-b} \right)^n
}}{{n!}}df^{\left( {n - 1} \right)} \left( t \right)} =-
\sum\limits_{k = 1}^n {\left( { - 1} \right)^k \frac{{\left( {z-b}
\right)^k }}{{k!}}f^{\left( {k - 1} \right)} \left( z \right) }  +
\int_z^b {f\left( t \right)dt}.\label{eq2.4}
\end{align}
Adding \eqref{eq2.2}--\eqref{eq2.4} we get the representation
\begin{multline}
\left( { - 1} \right)^n \int_a^b {S_n \left( {t,y,z}
\right)df^{\left( {n - 1} \right)} \left( t \right)}
\\
=  \sum\limits_{k = 1}^n {\frac{{\left( { - 1} \right)^k
}}{{k!}}\left[ {\left\{ {\left( {y - a} \right)^k  - \left( {y -
x} \right)^k } \right\}f^{\left( {k - 1} \right)} \left( y \right)
+ \left\{ {\left( {z - x} \right)^k  - \left( {z - b} \right)^k }
\right\}f^{\left( {k - 1} \right)} \left( z \right)} \right]}  +
\int_a^b {f\left( t \right)dt}. \label{eq2.5}
\end{multline}
A convenient presentation of \eqref{eq2.5} which gives an
expansion of real function $f$ near arbitrary two point $y,z\in
[a,b]$ with $y\le z$ can be written as:
\begin{align}
\left( {x - a} \right)f\left( y \right) + \left( {b - x}
\right)f\left( z \right)  =\int_a^b {f\left( t \right)dt} +
Q_n\left( {f,S_n;y,x,z} \right) + R_n\left( {f,S_n;y,x,z}
\right),\label{eq2.6}
\end{align}
for all $a \le y \le x \le z \le b$, where $Q_n\left( f;y,x,z
\right)$ is the two-point Ostrowski's formula given by
\begin{multline}
Q_n\left( {f,S_n;y,x,z} \right)\label{eq2.7}
\\
:=\sum\limits_{k = 2}^n {\frac{{\left( { - 1} \right)^k
}}{{k!}}\left[ {\left\{ {\left( {y - a} \right)^k - \left( {y - x}
\right)^k } \right\}f^{\left( {k - 1} \right)} \left( y \right) +
\left\{ {\left( {z - x} \right)^k - \left( {z - b} \right)^k }
\right\}f^{\left( {k - 1} \right)} \left( z \right)} \right]},
\end{multline}
and $R_n\left( {f,S_n;y,x,z} \right)$ is the remainder term given
such as
\begin{align}
\label{eq2.8}R_n\left( {f,S_n;y,x,z} \right):=\left( { - 1}
\right)^{n + 1} \int_a^b {S_n \left( {t;y,x,z} \right)df^{\left(
{n - 1} \right)} \left( t \right)},
\end{align}
for all $a\le  y \le x \le z \le b$.
If $n=1$, then the summation in \eqref{eq2.7} is defined to be  $0$.\\

\begin{remark}
\label{remark3}In \eqref{eq2.5} setting $a=y$, $b=z$ and replace
every $f$ by $f^{\prime}$, rearranging the terms gives the
generalized Taylor formula
\begin{align*}
f\left( z \right)   = f\left( y \right) + \sum\limits_{k = 1}^n
{\frac{1}{{k!}}\left[ {\left( {x - y} \right)^k f^{\left( {k}
\right)} \left( y \right)- \left( {x - z} \right)^kf^{\left( {k }
\right)} \left( z \right) } \right]} +\frac{1}{n!} \int_y^z
{\left( {x-t} \right)^n f^{\left( {n+1} \right)} \left( t
\right)dt}
\end{align*}
for all $x\in \left[y,z\right]$. This formula expand $f$ near two
point instead of one point. Moreover, if one chooses $y=a$ and
$z=x$ then we recapture the celebrated Taylor formula. From different  point of view, another Two-point formula was considered by Davis in (\cite{D}, p.37),  which completely independent formula. 
\end{remark}

\begin{theorem}
\label{thm4}If $f^{\left( {2n} \right)}$ $(n\ge1)$ is continuous
on $\left[a,b\right]$, then there exists $\eta \in
\left(a,b\right)$ such that
\begin{multline}
R_{2n}\left( f,S_{2n};y,x,z \right)
\\
=-\frac{1}{{\left( {2n + 1} \right)!}}\left[ {\left( {y - a}
\right)^{2n + 1}  + \left( {x - y} \right)^{2n + 1}  + \left( {z -
x} \right)^{2n + 1}  + \left( {b - z} \right)^{2n + 1} }
\right]\cdot f^{\left( {2n} \right)} \left( \eta
\right).\label{eq2.9}
\end{multline}
\end{theorem}
\begin{proof}
From \eqref{eq2.8} since $S_{2n} \left( {t;y,x,z} \right)$ does
not change sign over all $\left[a,b\right]$, the by Mean Value
theorem there is an $\eta \in \left(a,b\right)$ such that
\begin{align*}
&R_{2n}\left( {f,S_{2n};y,x,z} \right)
\\
&=\left( { - 1} \right)^{2n + 1} \int_a^b {S_{2n} \left( {t;y,x,z}
\right)f^{\left( {2n} \right)} \left( t \right)dt}
\\
&=- f^{\left( {2n} \right)} \left( \eta \right)\int_a^b {S_{2n}
\left( {t;y,x,z} \right)dt}
\\
&=-\frac{1}{{\left( {2n + 1} \right)!}}\left[ {\left( {y - a}
\right)^{2n + 1}  + \left( {x - y} \right)^{2n + 1}  + \left( {z -
x} \right)^{2n + 1}  + \left( {b - z} \right)^{2n + 1} }
\right]\cdot f^{\left( {2n} \right)} \left( \eta \right),
\end{align*}
which proves the result.
\end{proof}
In the previous theorem, since the error term for the Two-point
rule \eqref{eq2.9} involves $f^{\left(2n\right)}$, the rule gives
the exact result when applied to any function whose $(2n)$-th
derivative is identically zero, that is, any polynomial of degree
$\le 2n-1$, $n\in \mathbb{N}$.
\subsection{Remainder Estimates}

Let $f$ be defined on $[a,b]$, $P := \left\{ {x_0, x_1, \cdots,
x_n} \right\}$ be a partition of $[a,b]$, and
$$\Delta f_i  = f\left( {x_i } \right) - f\left( {x_{i - 1} }
\right),$$ for $i=1,2, \cdots, n$. A function $f$ is said to be of
bounded $p$-variation  if there exists a positive number $M$ such
that $\left( {\sum\limits_{i = 1}^n {\left| {\Delta f_i }
\right|^p} } \right)^{\frac{1}{p}}  \le M$, $(1 \le p < \infty)$
for all partitions of $[a,b]$.

Let $f$ be of bounded $p$-variation on $\left[a,b\right]$, and let
$\sum (P)$ denote the sum $\left( {\sum\limits_{i = 1}^n {\left|
{\Delta f_i } \right|^p} } \right)^{\frac{1}{p}} $ corresponding
to the partition $P$ of $[a,b]$. The number
\begin{align*}
\bigvee_{a}^{b}(f;p) = \sup\left\{ {\sum (P): P \in\mathcal{P}
\left( \left[ a,b \right]  \right)} \right\},\,\,\,\,\,\,\,\,\,\,1
\le p < \infty
\end{align*}
is called the total $p$--variation of $f$ on the interval $[a,b]$,
where $\mathcal{P}{\left( \left[ a,b \right]  \right) }$ denotes
the set of all partitions of $\left[ a,b \right]$. It can be
easily seen that for $p = 1$, $p$-variation  reduces to ordinary
variation or Jordan variation of functions.

In special case,   if there exists a positive number $M$ such that
\begin{align*}
\sum\limits_{i = 1}^n {  \osc\left( {f;\left[ {x_{i-1}^{\left( n
\right)} ,x_{i}^{\left( n \right)} } \right]}
\right)}=\sum\limits_{i = 1}^n {\left({  \sup  -  \inf
  }\right)  f\left( t_i \right) } \le M, \qquad  t_i  \in \left[ {x_{i-1}^{\left( n \right)} ,x_{i}^{\left( n
\right)} } \right],
\end{align*}
for all partitions of $\left[ a,b \right]$, then $f$ is said to
have $\infty$--variation on $\left[ a,b \right]$. The number
\begin{align*}
\bigvee_a^b(f;\infty) = \sup\left\{ {\sum (P): P \in
\mathcal{P}{[a,b]}} \right\}:=\osc\left( {f;\left[ {a,b} \right]}
\right),
\end{align*}
is called the oscillation of $f$ on $[a,b]$. Equivalently, we may
define the oscillation of $f$ such as, (see \cite{Dudley}):
\begin{align*}
\bigvee_a^b(f;\infty)= \mathop {\lim }\limits_{p \to \infty
}\bigvee_a^b(f;p)&= \mathop {\sup }\limits_{x \in \left[ {a,b}
\right]} \left\{ {f\left( x \right)} \right\} - \mathop {\inf
}\limits_{x \in \left[ {a,b} \right]} \left\{ {f\left( x \right)}
\right\}
\\
&= \osc\left( {f;\left[ {a,b} \right]} \right).
\end{align*}

If $f$ is a real function of bounded $p$-variation on an interval
$\left[ a,b \right]$, then  (\cite{alomari1}):
\begin{itemize}
\item $f$ is bounded, and
\begin{align*}
\osc\left( {f;\left[ {a,b} \right]} \right) \le
\bigvee_a^b(f;p)\le \bigvee_a^b(f;1).
\end{align*}
This fact follows by Jensen's inequality applied for
$h(p)=\bigvee_a^b(f;p)$ which is log-convex and decreasing for all
$p>1$. Moreover,   the
 inclusions
\begin{align*}
\mathcal{W}_{\infty}\left(f\right) \subset
\mathcal{W}_q\left(f\right) \subset
\mathcal{W}_p\left(f\right)\subset \mathcal{W}_1\left(f\right)
\end{align*}
are valid for all $1 < p < q < \infty$, where
$\mathcal{W}_p\left(\cdot\right)$ denotes the class of all
functions of bounded $p$-variation $(1\le p \le \infty)$ (see
\cite{Young}).

\item $f$ is continuous except at most on a countable set.

\item $f$ has one-sided limits everywhere (limits from the left
everywhere in $\left( a,b \right]$, and from the right everywhere
in $\left[ a,b \right)$;

\item The derivative $f^{\prime}(x)$ exists almost everywhere
(i.e. except for a set of measure zero).

\item If $f\left(x\right)$ is differentiable on $\left[ a,b
\right]$, then
\begin{align*}
\bigvee_a^b\left( {f;p} \right)  = \left( {\int_a^b {\left|
{f^{\prime}\left( t \right)} \right|^p dt} } \right)^{
\frac{1}{p}}=\left\| f^{\prime} \right\|_p, \qquad   1\le p <
\infty.
\end{align*}
\end{itemize}

In \cite{alomari1}, we find the following lemma:
\begin{lemma} \label{lemma1}
\emph{Fix $1 \le p < \infty$. Let $f,g : \left[ a,b \right] \to
\mathbb{R}$ be such that $f$ is continuous on $[a,b]$ and $g$ is
of bounded $p$--variation on $\left[ a,b \right]$. Then the
Riemann--Stieltjes integral $\int_a^b {f\left( t \right)dg\left( t
\right)}$ exists and the inequality:
\begin{align}
\label{eq2.10}\left| {\int_a^b {w\left( t \right)d\nu\left( t
\right)} } \right| \le \left\|{w}\right\|_{\infty}\cdot \osc\left(
{\nu;\left[ {a,b} \right]} \right)\le
\left\|{w}\right\|_{\infty}\cdot\bigvee_a^b\left( {\nu;p} \right),
\end{align}
holds.  The constant $`1$' in the both inequalities is the best
possible. }
\end{lemma}

In all next results, we need the following identities.
\begin{lemma} \label{lemma2}
    For the kernel $S_n:[a,b]^3\to\mathbb{R}$ $(n\ge1)$
given in \eqref{eq2.1}, and for all $a\le y \le x \le z \le b$, we
have the following computations:
\begin{align}
&\int_a^b {S_n \left( {t;y,x,z} \right)dt}
\label{eq2.11} \\
&= \frac{1}{{\left( {n + 1} \right)!}}\left[ {\left( {y - a}
    \right)^{n + 1}  + \left( {z - x} \right)^{n + 1}  + \left( { - 1}
    \right)^n \left( {x - y} \right)^{n + 1}  + \left( { - 1}
    \right)^n \left( {b - z} \right)^{n + 1} } \right].
\nonumber\\
\nonumber\\
&\int_a^b {\left| {S_n \left( {t;y,x,z} \right)} \right|dt}
\label{eq2.12}\\
&= \frac{1}{{\left( {n + 1} \right)!}}\left[ {\left( {y - a}
    \right)^{n + 1}  + \left( {x - y} \right)^{n + 1}  + \left( {z -
        x} \right)^{n + 1}  + \left( {b - z} \right)^{n + 1} }
\right].\nonumber
\\
\nonumber\\
&\int_a^b {\left|S_n \left( {t;y,x,z} \right)\right|^qdt}
\label{eq2.13}\\
&= \frac{1}{{\left( {nq + 1} \right)  n!}}\left[ {\left( {y - a}
    \right)^{nq + 1}  + \left( {x - y} \right)^{nq + 1}  + \left( {z -
        x} \right)^{nq + 1}  + \left( {b - z} \right)^{nq + 1} } \right],
\,\,\forall q\ge1.
\nonumber\\
\nonumber\\
&\mathop {\sup }\limits_{a \le t \le b} \left| {S_n \left(
    {t;y,x,z} \right)} \right| = \frac{1}{{n!}}\left[ {\max \left\{
    {\left( {y - a} \right),\left( {\frac{{z - y}}{2} + \left|
            {\frac{{y + z}}{2} - x} \right|} \right),\left( {b - z} \right)}
    \right\}} \right]^n.\label{eq2.14}
\end{align}
\end{lemma}
\begin{proof}
The proof is straightforward.
\end{proof}

Now we are ready to state our first result regarding the
estimation of the error term $R_n\left( f;y,x,z \right)$.
\begin{theorem}
\label{thm5}Let $I$ be a real interval such that $a,b$ in
$I^{\circ}$; the interior of $I$ with $a<b$. Let $f:I\to
\mathbb{R}$ be such that $f^{\left( {n - 1} \right)}$ is  of
bounded $p$-variation $(1\le p \le \infty)$ on $[a,b]\subset
I^{\circ}$, $\forall n\ge1$. Then
\begin{multline}
\left|{R_n\left( f,S_n;y,x,z \right)} \right|
\\
\le  \frac{1}{{n!}}\left[ {\max \left\{ {\left( {y - a}
\right),\left( {\frac{{z - y}}{2} + \left| {\frac{{y + z}}{2} - x}
\right|} \right),\left( {b - z} \right)} \right\}} \right]^n \cdot
\bigvee_a^b \left( {f^{\left( {n - 1} \right)} ,p}
\right),\label{eq2.15}
\end{multline}
for all $a \le y \le x \le z \le b$.

 Moreover, if $f^{\left(n\right)}$
exists then
\begin{align*} \bigvee_a^b\left( {f^{\left(n-1\right)};p} \right)  =
\left( {\int_a^b {\left| {f^{\left(n\right)}\left( t \right)}
\right|^p dt} } \right)^{ \frac{1}{p}}=\left\| f^{\left(n\right)}
\right\|_p, \qquad 1\le p < \infty,
\end{align*}
and therefore
\begin{multline}
\left|{R_n\left( f,S_n;y,x,z \right)} \right| \\\le
\frac{1}{{n!}}\left[ {\max \left\{ {\left( {y - a} \right),\left(
{\frac{{z - y}}{2} + \left| {\frac{{y + z}}{2} - x} \right|}
\right),\left( {b - z} \right)} \right\}} \right]^n \cdot \left\|
f^{\left(n\right)} \right\|_p,\label{eq2.16}
\end{multline}
\end{theorem}

\begin{proof}
Since $f^{\left( {n - 1} \right)}$ is  of bounded $p$-variation
$(1\le p \le \infty)$ on $\left[ a,b \right]\subset I^{\circ}$,
$\forall n\ge1$, utilizing the triangle integral inequality on the
identity \eqref{eq2.8} and employing Lemma \ref{lemma1}, we get
\begin{align*}
&\left| {\left( { - 1} \right)^{n+1} \int_a^b {S_n \left(
{t;y,x,z} \right)df^{\left( {n - 1} \right)} \left( t \right)} }
\right|
\\
&\le \mathop {\sup }\limits_{t \in \left[ {a,b} \right]} \left|
{S_n \left( {t;y,x,z} \right)} \right| \cdot \bigvee_a^b \left(
{f^{\left( {n - 1} \right)} ,p} \right)
 \\
&\le \frac{1}{{n!}}\left[ {\max \left\{ {\left( {y - a}
\right),\left( {\frac{{z - y}}{2} + \left| {\frac{{y + z}}{2} - x}
\right|} \right),\left( {b - z} \right)} \right\}} \right]^n \cdot
\bigvee_a^b \left( {f^{\left( {n - 1} \right)} ,p} \right),
\end{align*}
for all $p \in \left[1,\infty\right]$ and all $ n\ge1$, which
completes the proof. The moreover case follows from the properties
of $p$-variation and this completes the proof.
\end{proof}

\begin{theorem}
\label{thm6}Let $I$ be a real interval such that $a,b$ in
$I^{\circ}$; the interior of $I$ with $a<b$. Let $f:I\to
\mathbb{R}$ be such that $f^{\left( {n - 1} \right)}$ is
absolutely continuous    on $\left[a,b\right]\subset I^{\circ}$,
$\forall n\ge1$. Then  we have
\begin{multline}
\label{eq2.17}\left| R_n \left( {f,S_n;y,x,z} \right) \right|
\\
\le\left\{ \begin{array}{l}
 \frac{1}{{n!}}\left[ {\max \left\{ {\left( {y - a} \right),\left( {\frac{{z - y}}{2} + \left| {\frac{{y + z}}{2} - x} \right|} \right),\left( {b - z} \right)} \right\}} \right]^n  \cdot \left\| {f^{\left( n \right)} } \right\|_1,\\  \,\,\,\,\,\,\,\,\,\,\,\,\,\,\,\,\,\,\,\,\,\,\,\,\,\,\,\,\,\,\,\,\,\,\,\,\,\,\,\,\,\,\,\,\,\,\,\,\,\,\,\,\,\,\,\,\,\,\,\,\,\,\,\,\,\,\,\,\,\,\,\,\,\,\,\,\,\,\,\,\,\,\,\,\,\,\,\,\,\,\,\,\,\,\,\,\,\,\,\,\,\,\,\,\,\,\,\,\,\,\,\,\,\,\,\,\,\,\,\,\,\,\,\,\,\,\,\,\,\,\,\,\,\,\,\,\,\,\,\,\,\,\,\,\text{\emph{If}}\,\,f^{\left( n \right)}  \in L^1 \left( {\left[ {a,b} \right]} \right) \\
 \\
 \frac{1}{{\left( {nq + 1} \right)^{\frac{1}{q}} n!}}\left[ {\left( {y - a} \right)^{nq + 1}  + \left( {x - y} \right)^{nq + 1}  + \left( {z - x} \right)^{nq + 1}  + \left( {b - z} \right)^{nq + 1} } \right]^{\frac{1}{q}}  \cdot \left\| {f^{\left( n \right)} } \right\|_p,\\\,\,\,\,\,\,\,\,\,\,\,\,\,\,\,\,\,\,\,\,\,\,\,\,\,\,\,\,\,\,\,\,\,\,\,\,\,\,\,\,\,\,\,\,\,\,\,\,\,\,\,\,\,\,\,\,\,\,\,\,\,\,\,\,\,\,\,\,\,\,\,\,\,\,\,\,\,\,\,\,\,\,\,\,\,\,\,\,\,\,\,\,\,\,\,\,\,\,\,\,\,\,\,\,\,\,\,\,\,\,\,\,\,\,\,\,\,\,\,\,\,\,\,\,\,\,\,\,\,\,\,\,\,\,\,\,\,\,\,\,\,\,\,\text{\emph{If}}\,\,f^{\left( n \right)}  \in L^p \left( {\left[ {a,b} \right]} \right) \\
 \\
 \frac{1}{{\left( {n + 1} \right)!}}\left[ {\left( {y - a} \right)^{n + 1}  + \left( {x - y} \right)^{n + 1}  + \left( {z - x} \right)^{n + 1}  + \left( {b - z} \right)^{n + 1} } \right] \cdot \left\| {f^{\left( n \right)} } \right\|_\infty,\\\,\,\,\,\,\,\,\,\,\,\,\,\,\,\,\,\,\,\,\,\,\,\,\,\,\,\,\,\,\,\,\,\,\,\,\,\,\,\,\,\,\,\,\,\,\,\,\,\,\,\,\,\,\,\,\,\,\,\,\,\,\,\,\,\,\,\,\,\,\,\,\,\,\,\,\,\,\,\,\,\,\,\,\,\,\,\,\,\,\,\,\,\,\,\,\,\,\,\,\,\,\,\,\,\,\,\,\,\,\,\,\,\,\,\,\,\,\,\,\,\,\,\,\,\,\,\,\,\,\,\,\,\,\,\,\,\,\,\,\,\,\,\,\text{\emph{If}}\,\,f^{\left( n \right)}  \in L^\infty  \left( {\left[ {a,b} \right]} \right) \\
 \end{array} \right.,
\end{multline}
where $p>1$ with $\frac{1}{p}+\frac{1}{q}=1$. The inequality is
the best possible for $p = 1$, and sharp for $1 < p \le \infty$.
The equality is attained for every function $f\left( t \right) =
Mf_0 \left( t \right) + p_{n - 1} \left( t \right)$, $t\in
\left[a,b\right]$, where $M$ is real constant and $p_{n-1}$ is an
arbitrary polynomial of degree at most $n-1$, and $f_0$ is a
function defined on $\left[a,b\right]$ given by
\begin{align*}
f_0 \left( t \right) = \int_a^t {\frac{{\left( {t - u} \right)^{n
- 1} }}{{\left( {n - 1} \right)!}}{\mathop{\rm sgn}} S_n \left(
{u;y,x,z} \right) \cdot \left| {S_n \left( {u;y,x,z} \right)}
\right|^{\frac{1}{{p - 1}}} du}, \qquad t\in [a,b]
\end{align*}
for $1 < p < \infty$, and
\begin{align*}
f_0 \left( t \right) = \int_a^t {\frac{{\left( {t - u} \right)^{n
- 1} }}{{\left( {n - 1} \right)!}}{\mathop{\rm sgn}} S_n \left(
{u;y,x,z} \right) du}
\end{align*}
for $p=\infty$.
\end{theorem}

\begin{proof}
For $p=1$. Utilizing the triangle integral inequality on the
identity \eqref{eq2.8} then we have
\begin{align*}
&\left| {\left( { - 1} \right)^{n+1} \int_a^b {S_n \left(
{t;y,x,z} \right)f^{\left( n \right)} \left( t \right)dt} }
\right|
\\
&\le \mathop {\sup }\limits_{t \in \left[ {a,b} \right]} \left|
{S_n \left( {t;y,x,z} \right)} \right| \cdot \int_a^b {\left|
{f^{\left( n \right)} \left( t \right)} \right|dt}
\\
&= \frac{1}{{n!}}\left[ {\max \left\{ {\left( {y - a}
\right),\left( {\frac{{z - y}}{2} + \left| {\frac{{y + z}}{2} - x}
\right|} \right),\left( {b - z} \right)} \right\}} \right]^n \cdot
\left\| {f^{\left( n \right)} } \right\|_1.
\end{align*}
For $1 < p< \infty$, applying the H\"{o}lder integral inequality
on the identity \eqref{eq2.8} we get
\begin{align*}
&\left| {\left( { - 1} \right)^{n+1} \int_a^b {S_n \left( {t,y,z}
\right)f^{\left( n \right)} \left( t \right)dt} } \right| \\&\le
\left( {\int_a^b {\left| {f^{\left( n \right)} \left( t \right)}
\right|^p dt} } \right)^{\frac{1}{p}} \left( {\int_a^b {\left|
{S_n \left( {t;y,x,z} \right)} \right|^q dt} }
\right)^{\frac{1}{q}}
\\
&= \frac{1}{{\left( {nq + 1} \right)^{\frac{1}{q}} n!}}\left[
{\left( {y - a} \right)^{nq + 1}  + \left( {x - y} \right)^{nq +
1}  + \left( {z - x} \right)^{nq + 1}  + \left( {b - z}
\right)^{nq + 1} } \right]^{\frac{1}{q}}\cdot \left\| {f^{\left( n
\right)} } \right\|_p.
\end{align*}
For  $p=\infty$, we have
\begin{align*}
&\left| {\left( { - 1} \right)^{n+1} \int_a^b {S_n \left(
{t;y,x,z} \right)f^{\left( n \right)} \left( t \right)dt} }
\right| \\&\le \mathop {\sup }\limits_{t \in \left[ {a,b} \right]}
\left| {f^{\left( n \right)} \left( t \right)} \right| \cdot
\int_a^b {\left| {S_n \left( {t,y,z} \right)} \right|dt}
\\
&= \frac{1}{{\left( {n + 1} \right)!}}\left[ {\left( {y - a}
\right)^{n + 1}  + \left( {x - y} \right)^{n + 1}  + \left( {z -
x} \right)^{n + 1}  + \left( {b - z} \right)^{n + 1} }
\right]\cdot \left\| {f^{\left( n \right)} } \right\|_\infty,
\end{align*}
which proves the desired result. In order to prove the sharpness,
for $p=1$ we need to prove the first inequality in \eqref{eq2.17}
is the best possible. Suppose $\left| {S_n \left( {t;y,x,z}
\right)} \right|$ attains its supremum at point $t_0 \in \left[a,
b\right]$ and let $\mathop {\sup }\limits_{t \in \left[ {a,b}
\right]} \left| {S_n \left( {t;y,x,z} \right)} \right| = S_n
\left( {t_0,y,x,z} \right)$ for some $k = 1, 2, 3$.

Let
$\mathcal{A}_-=:\left\{\left(a,y\right],\left(y,z\right],\left(z,b\right]\right\}$
and assume that $S_n \left( {t_0,y,x,z} \right)>0$. For $\epsilon$
small enough define $f_{\epsilon}^{\left( {n-1} \right)} \left(
{t} \right)$ by
\begin{align*}
f_{\epsilon}^{\left( n-1 \right)} \left( t \right) = \left\{
\begin{array}{l}
0,\,\,\,\,\,\,\,\,\,\,\,\,\,\, \,\,\,\,\,\,t \le t_0 -\epsilon\\
 \frac{t - t_0 +\epsilon}{\epsilon},\,\,\,\,\,\,t \in \left[ {t_0  - \varepsilon ,t_0 } \right] \\
 1,\,\,\,\,\,\,\,\,\,\,\,\,\,\,\,\,\,\,\,\,t>t_0 \\
 \end{array} \right..
\end{align*}
 If $t_0 \in \left(c,d\right] \in \mathcal{A}_-$. Then, for $\epsilon$ small enough,
\begin{align*}
\left| {\int_a^b {S_n\left( {t;y,x,z} \right)f_{\epsilon}^{\left(
n \right)} \left( t \right)dt} } \right| = \frac{1}{\epsilon
}\left| {\int_{t_0 - \epsilon }^{t_0 } {S_n \left( {t;y,x,z}
\right)dt} } \right| &\le \frac{1}{\epsilon }\int_{t_0 - \epsilon
}^{t_0 } { S_n \left( {t;y,x,z} \right) dt}
\\
&\le \mathop {\sup }\limits_{t_0  - \epsilon  \le t \le t_0 }
\left| {S_n\left( {t;y,x,z} \right)} \right| \cdot
\frac{1}{\epsilon }\int_{t_0  - \epsilon }^{t_0 } {dt}
\\
&= S_n \left( {t_0,y,x,z} \right),
\end{align*}
also since $\mathop {\lim }\limits_{\epsilon \to 0}
\frac{1}{\epsilon}\int_{t_0 - \epsilon}^{t_0 } {S_n \left(
{t;y,x,z} \right)dt}  = S_n \left( {t_0,y,x,z} \right)$, the
statement follows.

For $t_0=c\in \left\{a,y,z\right\}$, then we define, for $\epsilon
> 0$ small enough, function $f_{\epsilon}^{\left( {n-1} \right)} \left( {t} \right)$ by
\begin{align*}
f_{\epsilon}^{\left( n-1 \right)} \left( t \right) = \left\{
\begin{array}{l}
0,\,\,\,\,\,\,\,\,\,\,\,\,\,\, \,\,\,\,\,\,t \le t_0\\
 \frac{t - t_0 }{\epsilon},\,\,\,\,\,\,t \in \left[ {t_0,t_0+\epsilon } \right] \\
 1,\,\,\,\,\,\,\,\,\,\,\,\,\,\,\,\,\,\,\,\,t>t_0+\epsilon \\
 \end{array} \right.,
\end{align*}
and we continue as above. In similar manner we can show the
sharpness holds when $S_n \left( {t_0,y,x,z} \right)<0$.

Finally, for $1<p\le \infty$, the function $f_0\left(t\right)$
given in the Theorem \ref{thm6} is $n$-times differentiable, and
its $n$-th derivative is piecewise continuous function. Further,
$f_0$ is a solution of the differential equation
\begin{align*}
S_n \left( {t;y,x,z} \right)f^{\left( n \right)} \left( t \right)
= \left| {S_n \left( {t;y,x,z} \right)} \right|^q, \qquad q=
\frac{p}{p-1}, \forall p>1,
\end{align*}
so the equality in \eqref{eq2.17}  holds for $1<p\le \infty$.
\end{proof}

To treat bounds for functions possess H\"{o}lder continuity of
order $r \in \left(0,1\right]$. Let $t_0\in \left[a,b\right]$ be
fixed point. From \eqref{eq2.8}, we have
\begin{align}
&R_n\left( f,S_n;y,x,z \right)
\nonumber\\
&=\left( { - 1} \right)^{n+1} \int_a^b {S_n \left( {t;y,x,z}
\right)df^{\left( {n-1} \right)} \left( t \right)}
\nonumber\\
&=\left( { - 1} \right)^{n+1} \int_a^b {S_n \left( {t;y,x,z}
\right)d\left[{f^{\left( {n-1} \right)} \left( t \right)-f^{\left(
{n-1} \right)} \left( t_0 \right) }\right]}
\nonumber\\
&= \frac{\left( { - 1} \right)^{n}}{n!}\left( {y-a}
\right)^n\left[{f^{\left( {n-1} \right)} \left( y
\right)-f^{\left( {n-1} \right)} \left( t_0\right)
}\right]+\frac{\left( { - 1} \right)^{n}}{n!}\left( {z-x}
\right)^n\left[{f^{\left( {n-1} \right)} \left( z
\right)-f^{\left( {n-1} \right)} \left( t_0\right) }\right]
\nonumber\\
&\qquad-\frac{\left( { - 1} \right)^{n}}{n!}\left( {y-x}
\right)^n\left[{f^{\left( {n-1} \right)} \left( y
\right)-f^{\left( {n-1} \right)} \left( t_0\right)
}\right]-\frac{\left( { - 1} \right)^{n}}{n!}\left( {z-b}
\right)^n\left[{f^{\left( {n-1} \right)} \left( z
\right)-f^{\left( {n-1} \right)} \left( t_0\right) }\right]
\nonumber\\
&\qquad+ \left( { - 1} \right)^{n} \int_a^b {\left[{f^{\left(
{n-1} \right)} \left( t \right)-f^{\left( {n-1} \right)} \left(
t_0\right) }\right]dS_{n} \left( {t;y,x,z} \right)}
\nonumber\\
&=   \frac{\left( { - 1} \right)^{n}}{n!}\left[{\left( {y-a}
\right)^n -\left( {y-x} \right)^n}\right]\left[{f^{\left( {n-1}
\right)} \left( y \right)-f^{\left( {n-1} \right)} \left(
t_0\right) }\right]\label{eq2.18}
 \\
&\qquad+\frac{\left( { - 1} \right)^{n}}{n!}\left[{\left( {z-x}
\right)^n-\left( {z-b} \right)^n }\right]\left[{f^{\left( {n-1}
\right)} \left( z \right)-f^{\left( {n-1} \right)} \left(
t_0\right) }\right]
\nonumber\\
&\qquad+ \left( { - 1} \right)^{n} \int_a^b {\left[{f^{\left(
{n-1} \right)} \left( t \right)-f^{\left( {n-1} \right)} \left(
t_0\right) }\right]dS_{n} \left( {t;y,x,z} \right)}.\nonumber
\end{align}
Now, let us setting
\begin{align}
\label{eq2.19} \widetilde{R}_n\left( f,S_n;y,x,z \right)= \left( {
- 1} \right)^{n} \int_a^b {\left[{f^{\left( {n-1} \right)} \left(
t \right)-f^{\left( {n-1} \right)} \left( t_0\right)
}\right]dS_{n} \left( {t;y,x,z} \right)}.
\end{align}
\begin{theorem}
\label{thm7} Let $I$ be a real interval such that $a,b$ in
    $I^{\circ}$; the interior of $I$ with $a<b$. Let $f:I\to
    \mathbb{R}$ be such that $f^{\left( {n - 1} \right)}$ satisfy the
    H\"{o}lder condition with exponent $r\in \left(0,1\right]$ and
    constant $H>0$ on $\left[a,b\right]\subset I^{\circ}$, $\forall n\ge1$,
    then we have
    \begin{align}
\label{eq2.20} \left|\widetilde{R}_n\left( f,S_n;y,x,z
\right)\right|\le
     H
    \left\{ \begin{array}{l}
    \frac{\left({b-t_0}\right)^{r+1}+\left({t_0-a}\right)^{r+1}}{r+1} \cdot \left\| {S_{n - 1} \left( {\cdot;y,x,z} \right)} \right\|_{\infty}
    \\
    \\
    \left(\frac{\left({b-t_0}\right)^{pr+1}+\left({t_0-a}\right)^{pr+1}}{pr+1}\right)^{1/p}\cdot \left\| {S_{n - 1} \left( {\cdot;y,x,z} \right)} \right\|_{q}
    \\
    \\
    \left[ {\frac{{b - a}}{2} + \left| {t_0  - \frac{{a + b}}{2}} \right|} \right]^r \cdot \left\| {S_{n - 1} \left( {\cdot;y,x,z} \right)} \right\|_{1}\\
    \end{array} \right.
    \end{align}
    for all  $x\in \left[a,b\right]$ and $p>1$ with
    $\frac{1}{p}+\frac{1}{q}=1$.
\end{theorem}
\begin{proof}
Since  $f^{(n-1)}$ is $r$--H\"{o}lder continuous on
$\left[a,b\right]$, then
\begin{align*}
&\left|\widetilde{R}_n\left( f,S_n;y,x,z \right)\right|
\\
&= \left|\left( { - 1} \right)^{n} \int_a^b {\left[{f^{\left(
{n-1}\right)} \left( t \right)-f^{\left( {n-1} \right)} \left(
t_0\right) }\right]dS_{n} \left( {t;y,x,z} \right)}\right|
\\
&\le  \int_a^b {\left|{f^{\left( {n-1}\right)} \left( t
\right)-f^{\left( {n-1} \right)} \left(t_0\right)
}\right|\left|S_{n-1} \left( {t;y,x,z} \right) \right|dt}
\\
&\le H\int_a^b {\left| {t - t_0 } \right|^r \left| {S_{n - 1}
\left( {t;y,x,z} \right)} \right|dt}
\\
&\le H\cdot \left\{ \begin{array}{l}
\mathop {\sup }\limits_{t \in \left[ {a,b} \right]} \left\{ {\left| {S_{n - 1} \left( {t;y,x,z} \right)} \right|} \right\}\int_a^b {\left| {t - t_0 } \right|^r dt}  \\
\\
\left( {\int_a^b {\left| {t - t_0 } \right|^{rp} dt} } \right)^{1/p} \left( {\int_a^b {\left| {S_{n - 1} \left( {t;y,x,z} \right)} \right|^q dt} } \right)^{1/q}  \\
\\
\mathop {\sup }\limits_{t \in \left[ {a,b} \right]} \left\{ {\left| {t - t_0 } \right|^r } \right\}\int_a^b {\left| {S_{n - 1} \left( {t;y,x,z} \right)} \right|dt}  \\
\end{array} \right.
\\
&= H \left\{ \begin{array}{l}
 \frac{\left({b-t_0}\right)^{r+1}+\left({t_0-a}\right)^{r+1}}{r+1} \cdot \left\| {S_{n - 1} \left( {\cdot;y,x,z} \right)} \right\|_{\infty}
\\
\\
  \left(\frac{\left({b-t_0}\right)^{pr+1}+\left({t_0-a}\right)^{pr+1}}{pr+1}\right)^{1/p}\cdot \left\| {S_{n - 1} \left( {\cdot;y,x,z} \right)} \right\|_{q}
\\
\\
\left[ {\frac{{b - a}}{2} + \left| {t_0  - \frac{{a + b}}{2}} \right|} \right]^r \cdot \left\| {S_{n - 1} \left( {\cdot;y,x,z} \right)} \right\|_{1}\\
\end{array} \right.
\end{align*}
for all $n\ge1$ and for every $a\le y \le x \le z \le b$ with all
$t_0\in\left[a,b\right]$.
\end{proof}

\begin{remark}
In very interesting case, one may choose $t_0=x$ in
\eqref{eq2.18}--\eqref{eq2.19}, so that \eqref{eq2.20} becomes
    \begin{align}
\label{eq2.21}\left|\widetilde{R}_n\left( f,S_n;y,x,z
\right)\right|\le H \left\{ \begin{array}{l}
\frac{\left({b-x}\right)^{r+1}+\left({ x-a}\right)^{r+1}}{r+1}
\cdot \left\| {S_{n - 1} \left( {\cdot;y,x,z} \right)}
\right\|_{\infty}
\\
\\
\left(\frac{\left({b-x}\right)^{pr+1}+\left({x-a}\right)^{pr+1}}{pr+1}\right)^{1/p}\cdot
\left\| {S_{n - 1} \left( {\cdot;y,x,z} \right)} \right\|_{q}
\\
\\
\left[ {\frac{{b - a}}{2} + \left| {x  - \frac{{a + b}}{2}} \right|} \right]^r \cdot \left\| {S_{n - 1} \left( {\cdot;y,x,z} \right)} \right\|_{1}\\
\end{array} \right.
\end{align}
for all  $a \le y \le x \le z \le b$,  $r\in \left( 0,1\right]$,
and all $p>1$ with $\frac{1}{p}+\frac{1}{q}=1$.
\end{remark}

\begin{corollary}
Let $I$ be a real interval such that $a,b$ in $I^{\circ}$; the
interior of $I$ with $a<b$. Let $f:I\to \mathbb{R}$ be such that
$f^{\left( {n - 1} \right)}$ satisfy the H\"{o}lder condition with
exponent $r\in \left(0,1\right]$ and constant $H>0$ on
$\left[a,b\right]\subset I^{\circ}$, $\forall n\ge1$, then we have
\begin{multline}
\label{eq2.22}\left|\widetilde{R}_n\left(
f,S_n;h,\frac{a+b}{2},a+b-h \right)\right|
\\
\le H   \left\{ \begin{array}{l} \frac{
\left({b-a}\right)^{r+1}}{2^r(r+1)} \cdot \left\| {S_{n - 1}
\left( {\cdot;h,\frac{a+b}{2},a+b-h} \right)} \right\|_{\infty}
\\
\\
\left(\frac{\left({b-a}\right)^{pr+1}}{2^{pr+1}(pr+1)}\right)^{1/p}\cdot
\left\| {S_{n - 1} \left( {\cdot;h,\frac{a+b}{2},a+b-h} \right)}
\right\|_{q}
\\
\\
\left( {\frac{{b - a}}{2}} \right)^r \cdot \left\| {S_{n - 1} \left( {\cdot;h,\frac{a+b}{2},a+b-h} \right)} \right\|_{1}\\
\end{array} \right.
\end{multline}
for all  $h\in \left[a,\frac{a+b}{2}\right]$,  $r\in \left(
0,1\right]$, and all $p>1$ with $\frac{1}{p}+\frac{1}{q}=1$.

Also, we have
\begin{align}
\label{eq2.23}\left|\widetilde{R}_n\left( f,S_n;a,x,b
\right)\right|\le H \left\{ \begin{array}{l}
\frac{\left({b-x}\right)^{r+1}+\left({ x-a}\right)^{r+1}}{r+1}
\cdot \left\| {S_{n - 1} \left( {\cdot;a,x,b} \right)}
\right\|_{\infty}
\\
\\
\left(\frac{\left({b-x}\right)^{pr+1}+\left({x-a}\right)^{pr+1}}{pr+1}\right)^{1/p}\cdot
\left\| {S_{n - 1} \left( {\cdot;a,x,b} \right)} \right\|_{q}
\\
\\
\left[ {\frac{{b - a}}{2} + \left| {x  - \frac{{a + b}}{2}} \right|} \right]^r \cdot \left\| {S_{n - 1} \left( {\cdot;a,x,b} \right)} \right\|_{1}\\
\end{array} \right.
\end{align}
for all  $x\in \left[a,b\right]$,  $r\in \left( 0,1\right]$, and
all $p>1$ with $\frac{1}{p}+\frac{1}{q}=1$. Similarly, for
$\left|\widetilde{R}_n\left( f,S_n;x,x,x \right)\right|$ the bound
in \eqref{eq2.23} holds.
\end{corollary}

Clearly, $\widetilde{R}_n\left( f,S_n;y,x,z \right)=R_n\left(
f,S_n;y,x,z \right)$, when $y=z=x=t_0$. Therefore we may state the
following bound for mappings $f^{\left( {n - 1} \right)}$ satisfy
the H\"{o}lder condition.
\begin{corollary}
Let $I$ be a real interval such that $a,b$ in $I^{\circ}$; the
interior of $I$ with $a<b$. Let $f:I\to \mathbb{R}$ be such that
$f^{\left( {n - 1} \right)}$ satisfy the H\"{o}lder condition with
exponent $r\in \left(0,1\right]$ and constant $K>0$ on
$\left[a,b\right]\subset I$, $\forall n\ge1$, then we have
\begin{align*}
\left|{R_n\left( f,S_n;x,x,x\right)} \right| \le H \frac{{\Gamma
\left( {1 + r} \right)}}{{\Gamma \left( {1 + n + r} \right)}}
 \left[ {\left( {x - a} \right)^{r + n}  + \left( {b - x}
\right)^{r  + n} } \right],
\end{align*}
for all $x \in \left[a,b\right]$ and $r \in \left(0,1\right]$.
Moreover, if $f^{\left(n-1\right)}$ satisfies the Lipschitz
condition with constant $L$, i.e., $r=1$, then we have
\begin{align*}
 \left|{R_n\left( f,S_n;x,x,x\right)} \right|
\le  \frac{L}{{\left( {n+1} \right)!}} \left[ {\left( {x - a}
\right)^{ n+1}  + \left( {b - x} \right)^{n+1} } \right].
\end{align*}
\end{corollary}

\begin{proof}
Now, since $t_0$ is arbitrarily chosen in $\left[a,b\right]$, we
give the following detailed estimates of $\widetilde{R}_n\left(
f,S_n;y,x,z \right) $. From \eqref{eq2.19}, we can write
\begin{align*}
&\left|{\widetilde{R}_n\left( f,S_n;x,x,x \right) }\right|
\nonumber\\
&\le \frac{1}{\left(n-1\right)!} \int_a^x {\left|{f^{\left( {n-1}
\right)} \left( t \right)-f^{\left( {n-1} \right)} \left(
t_0\right) }\right|\left( {t-a} \right)^{n-1}dt}
\nonumber \\
 &\qquad\qquad+ \frac{1}{\left(n-1\right)!} \int_x^b
{\left|{f^{\left( {n-1} \right)} \left( t \right)-f^{\left( {n-1}
\right)} \left( t_0\right) }\right|\left( {b-t} \right)^{n-1}dt}
\nonumber\\
&\le \frac{H}{\left(n-1\right)!} \left[{  \int_a^x {\left|{t-x
}\right|^{r}\left( {t-a} \right)^{n-1}dt}   + \int_x^b {\left|{t-x
}\right|^{r}\left( {b-t} \right)^{n-1}dt} }\right]
\\
&=H \frac{{r \Gamma \left( r  \right)}}{{\left( {r + n}
\right)\Gamma \left( {r  + n} \right)}} \left[ {\left( {x - a}
\right)^{r + n}  + \left( {b - x} \right)^{r  + n} } \right],
\end{align*}
and this ends the proof.
\end{proof}

%=============================================================================
\section{Two--point Ostrowski formula via Fink approach}\label{sec3}
%=============================================================================
In this sections we present a Fink type identity for two-point
formula \eqref{eq1.14} and then we give some estimates of the
remainder.

\subsection{General Fink type identity}
An identity which express two-point formula of Ostrowski's type
via Fink approach and using harmonic sequence of polynomials can
be given by the representation:
\begin{theorem}\label{thm8}
Let $I$ be a real interval, $a,b \in I^{\circ}$ $(a<b)$. Let $Q_k$
be a harmonic sequence of polynomials and let $f:I \to \mathbb{R}$
be such that $f^{\left( n \right)}$ is absolutely continuous on
$I$ for $n\ge1$ with $Q_{n-1}\left( {t} \right) S\left( {t,\cdot}
\right)f^{\left( n \right)} \left( t \right)$ is integrable. Then
we have the representation
\begin{multline}
\frac{1}{n}\left[ {\frac{\left( {x - a} \right)f\left( y \right) +
\left( {b - x} \right)f\left( z \right)}{b-a} + \sum\limits_{k =
1}^{n - 1} { \left\{{T_k\left(y,x,z\right)+F_k \left( {a,b}
\right) }\right\}} } \right] - \frac{1}{b-a}\int_a^b {f\left( s
\right)ds}
\\
= \frac{{\left( { - 1} \right)^{n - 1} }}{n\left(b-a
\right)}\int_a^b {Q_{n - 1} \left( t \right)K\left( {t;y,x,z}
\right)f^{\left( n \right)} \left( t \right)dt},\label{eq3.1}
\end{multline}
for all $a\le y \le x \le z \le b$, where
\begin{align}
T_k\left(y,x,z\right):=\frac{\left( { - 1} \right)^k}{b-a} \left\{
{\left( {x - a} \right)Q_k \left( y \right)f^{\left( k \right)}
\left( y \right)+
  \left( {b - x} \right)Q_k \left( z \right)f^{\left( k \right)}
\left( z \right)} \right\},\label{eq3.2}
\end{align}
\begin{align}
F_k\left({a,b}\right) = \frac{\left( { - 1} \right)^k \left( {n -
k} \right)}{\left( {b - a} \right)}\left[ {Q_k \left( a
\right)f^{\left( {k - 1} \right)} \left( a \right) - Q_k \left( b
\right)f^{\left( {k - 1} \right)} \left( b \right)} \right],
\label{eq3.3}
\end{align}
 and
\begin{align}
 K\left( {t;y,x,z}
\right)= \left\{ \begin{array}{l}
 t - a,\,\,\,\,\,\,\,\,\,\,\,\,\,\,\,\,\,\,\,\,\,\,\,\,\,\,\,\,\,\,\,\,\,\,\,\,a \le t \le
 y
 \\
 t - x,\,\,\,\,\,\,\,\,\,\,\,\,\,\,\,\,\,\,\,\,\,\,\,\,\,\,\,\,\,\, y  < t <
 z
 \\
 t - b,\,\,\,\,\,\,\,\,\,\,\,\,\,\,\,\,\, z \le t \le b
 \end{array} \right.,\label{eq3.4}
 \end{align}
 for all $a\le y \le x \le z \le b$.
 \end{theorem}

\begin{proof}
From \eqref{eq1.11}, we find the formula
\begin{multline}
\frac{1}{n}\left[ {f\left( x \right) + \sum\limits_{k = 1}^{n - 1}
{\left( { - 1} \right)^k Q_k \left( x \right)f^{\left( k \right)}
\left( x \right)}  + \sum\limits_{k = 1}^{n - 1} {F_k \left( {a,b}
\right)} } \right] - \frac{1}{{b - a}}\int_a^b {f\left( y
\right)dy}
 \\
= \frac{{\left( { - 1} \right)^{n - 1} }}{{\left( {b - a}
\right)n}}\int_a^b {\left( {\int_s^x {Q_{n - 1} \left( t
\right)f^{\left( n \right)} \left( t \right)dt} }
\right)ds}\label{eq3.5}
\end{multline}
where
\begin{align}
F_k\left( {a,b} \right) = \frac{{\left( { - 1} \right)^k \left( {n
- k} \right)}}{{b - a}}\left[ {Q_k \left( a \right)f^{\left( {k -
1} \right)} \left( a \right) - Q_k \left( b \right)f^{\left( {k -
1} \right)} \left( b \right)} \right].\label{eq3.6}
\end{align}
Fix $y,z\in \left[a,b\right]$. In the representation
\eqref{eq3.5}, replace $x$ by $y$ and $b$ by $x$ we get
\begin{multline}
\frac{1}{n}\left[ {f\left( y \right) + \sum\limits_{k = 1}^{n - 1}
{\left( { - 1} \right)^k Q_k \left( y \right)f^{\left( k \right)}
\left( y \right)}  + \sum\limits_{k = 1}^{n - 1} {F_k \left( {a,x}
\right)} } \right] - \frac{1}{{x - a}}\int_a^x {f\left( s
\right)ds}
 \\
= \frac{{\left( { - 1} \right)^{n - 1} }}{{\left( {x - a}
\right)n}}\int_a^x {\left( {\int_s^y {Q_{n - 1} \left( t
\right)f^{\left( n \right)} \left( t \right)dt} }
\right)ds}.\label{eq3.7}
\end{multline}
Multiplying both sides of \eqref{eq3.7} by $\left(x-a\right)$, we
get
\begin{multline}
\frac{{\left( {x - a} \right)}}{n}\left[ {f\left( y \right) +
\sum\limits_{k = 1}^{n - 1} {\left( { - 1} \right)^k Q_k \left( y
\right)f^{\left( k \right)} \left( y \right)}  + \sum\limits_{k =
1}^{n - 1} {F_k \left( {a,x} \right)} } \right] - \int_a^x
{f\left(s \right)ds}
\\
= \frac{{\left( { - 1} \right)^{n - 1} }}{n}\int_a^x {\left(
{\int_s^y {Q_{n - 1} \left( t \right)f^{\left( n \right)} \left( t
\right)dt} } \right)ds}\label{eq3.8}
\end{multline}
The second step, in the same formula \eqref{eq3.5} we replace
every $x$ by $z$ and $a$ by $x$, then $f$ has the representation
\begin{multline}
\frac{1}{n}\left[ {f\left( z \right) + \sum\limits_{k = 1}^{n - 1}
{\left( { - 1} \right)^k Q_k \left( z \right)f^{\left( k \right)}
\left( z \right)}  + \sum\limits_{k = 1}^{n - 1} {F_k \left( {x,b}
\right)} } \right] - \frac{1}{{b - x}}\int_x^b {f\left( s
\right)ds}
\\
= \frac{{\left( { - 1} \right)^{n - 1} }}{{\left( {b - x}
\right)n}}\int_x^b {\left( {\int_s^z {Q_{n - 1} \left( t
\right)f^{\left( n \right)} \left( t \right)dt} }
\right)ds},\label{eq3.9}
\end{multline}
Multiplying both sides of \eqref{eq3.9} by $\left(b-x\right)$ we
get
\begin{multline}
\frac{{\left( {b - x} \right)}}{n}\left[ {f\left( z \right) +
\sum\limits_{k = 1}^{n - 1} {\left( { - 1} \right)^k Q_k \left( z
\right)f^{\left( k \right)} \left( z \right)}  + \sum\limits_{k =
1}^{n - 1} { F_k  \left( {x,b} \right)} } \right] - \int_x^b
{f\left( s \right)ds}
\\
= \frac{{\left( { - 1} \right)^{n - 1} }}{n}\int_x^b {\left(
{\int_s^z {Q_{n - 1} \left( t \right)f^{\left( n \right)} \left( t
\right)dt} } \right)ds}\label{eq3.10}
\end{multline}
Adding the equations \eqref{eq3.8} and \eqref{eq3.10} we get
\begin{multline}
\frac{1}{n}\left[ {\left( {x - a} \right)f\left( y \right) +
\left( {b - x} \right)f\left( z \right) + \sum\limits_{k = 1}^{n -
1} {\left( { - 1} \right)^k \left\{ {\left( {x - a} \right)Q_k
\left( y \right)f^{\left( k \right)} \left( y \right) + \left( {b
- x} \right)Q_k \left( z \right)f^{\left( k \right)} \left( z
\right)} \right\}} } \right]
\\
+ \left[ {\left( {x - a} \right)\sum\limits_{k = 1}^{n - 1} { F_k
\left( {a,x} \right)}  + \left( {b - x} \right)\sum\limits_{k =
1}^{n - 1} {F_k \left( {x,b} \right)} } \right] - \int_a^b
{f\left( s \right)ds}
\\
= \frac{{\left( { - 1} \right)^{n - 1} }}{n}\left[ {\int_a^x
{\left( {\int_s^y {Q_{n - 1} \left( t \right)f^{\left( n \right)}
\left( t \right)dt} } \right)ds}  + \int_x^b {\left( {\int_s^z
{Q_{n - 1} \left( t \right)f^{\left( n \right)} \left( t
\right)dt} } \right)ds} } \right] \label{eq3.11}
\end{multline}
To simplify the right hand side, we write
\begin{align}
\int_a^{x} {ds} \int_s^y {dt}  = \int_a^y {ds} \int_s^y {dt}  +
\int_y^{x} {ds} \int_s^y {dt} &= \int_a^y {dt} \int_a^t {ds}  -
\int_y^{x} {ds} \int_y^s {dt}
\nonumber\\
&= \int_a^y {dt} \int_a^t {ds}  - \int_y^{x} {dt} \int_t^{x} {ds},
\label{eq3.12}
\end{align}
and
\begin{align}
\int_{x}^b {ds} \int_s^{z} {dt}  = \int_{x}^{z} {ds} \int_s^{z}
{dt} + \int_{z}^b {ds} \int_s^{z} {dt} &= \int_{x}^{z} {dt}
\int_{x}^t {ds}  - \int_{x}^b {ds} \int_z^{s} {dt}
\nonumber\\
&= \int_{x}^{z} {dt} \int_{x}^t {ds}  - \int_{z}^b {dt} \int_t^b
{ds}.\label{eq3.13}
\end{align}
Adding \eqref{eq3.12} and \eqref{eq3.13}, we get
\begin{multline}
\int_a^{x} {dy} \int_s^y {dt}+\int_{x}^b {ds} \int_s^{z} {dt}
\\
= \int_a^y {dt} \int_a^t {ds}  - \int_y^{x} {dt} \int_t^{x}
{ds}+\int_{x}^{z} {dt} \int_{x}^t {ds}  - \int_{z}^b {dt} \int_t^b
{ds}.\label{eq3.14}
\end{multline}
In viewing of \eqref{eq3.14}, the right hand side of
\eqref{eq3.11} is simplified to be
\begin{align*}
&\frac{{\left( { - 1} \right)^{n - 1} }}{n}\left[ {\int_a^x
{\left( {\int_s^y {Q_{n - 1} \left( t \right)f^{\left( n \right)}
\left( t \right)dt} } \right)ds}  + \int_x^b {\left( {\int_s^z
{Q_{n - 1} \left( t \right)f^{\left( n \right)} \left( t
\right)dt} } \right)ds} } \right]
\\
&= \frac{{\left( { - 1} \right)^{n - 1} }}{n}\int_a^b {Q_{n - 1}
\left( t \right)K\left( {t;y,x,z} \right)f^{\left( n \right)}
\left( t \right)dt},
\end{align*}
where
\begin{align*}
 K\left( {t;y,x,z } \right) = \left\{ \begin{array}{l}
 t - a,\,\,\,\,\,\,\,\,\,\,a \le t \le y
 \\
 t - x ,\,\,\,\,\,\,\,\,\,\,y  < t < z
 \\
 t - b,\,\,\,\,\,\,\,\,\,\,z \le t \le b
 \end{array} \right.,
 \end{align*}
 for all $a \le y \le x \le z \le b$.

Also, we note that in the right hand side
\begin{align*}
\left( {x - a} \right)F_k \left( {a,x} \right) + \left( {b - x}
\right)F_k \left( {x,b} \right)
 &= \left( { - 1} \right)^k \left( {n - k} \right)\left[ {Q_k \left( a \right)f^{\left( {k - 1} \right)} \left( a \right) - Q_k \left( b \right)f^{\left( {k - 1} \right)} \left( b \right)} \right]
\\
&= \left( {b - a} \right)F_k \left( {a,b} \right).
\end{align*}
Hence, the identity \eqref{eq3.3} is obtained by combining the
last two equalities with \eqref{eq3.6}, and this ends the proof.
\end{proof}

%\begin{remark}
%Setting $x=y=z$ in \eqref{eq3.1}, we refer to Dedi\'{c} \etal
%representation \eqref{?}.
%\end{remark}

\subsection{Error Estimates}
To approximate $\int_a^b{f\left(t\right)dt} $, let us define the
general quadrature rule
\begin{align}
\label{eq3.15} \int_a^b{f\left(t\right)dt} =
\mathcal{G}_n\left(f,Q_n;y,x,z\right) +
\mathcal{E}_n\left(f,Q_n;y,x,z\right),
\end{align}
for all  $a\le y\le x \le z \le b$, where
$\mathcal{G}_n\left(f,Q_n;y,x,z\right)$ is the quadrature formula
given by
\begin{multline}
\mathcal{G}_n\left(f,Q_n;y,x,z\right)
\\
= \frac{b-a}{n}\left[ {\frac{\left( {x - a} \right)f\left( y
\right) + \left( {b - x} \right)f\left( z \right)}{b-a} +
\sum\limits_{k = 1}^{n - 1} { \left\{{T_k\left(y,x,z\right)+F_k
\left( {a,b} \right) }\right\}} } \right],\label{eq3.16}
\end{multline}
and $\mathcal{E}_n\left(f,Q_n;y,x,z\right)$ is the error term
given by
\begin{align}
\mathcal{E}_n\left(f,Q_n;y,x,z\right)= \frac{{\left( { - 1}
\right)^{n} }}{n}\int_a^b {Q_{n - 1} \left( t \right)K\left(
{t;y,x,z} \right)f^{\left( n \right)} \left( t
\right)dt}.\label{eq3.17}
\end{align}

\begin{theorem}
\label{thm9} Under the assumptions of Theorem \ref{thm8}, we have
\begin{align}
\left|{\mathcal{E}_n\left(f,Q_n;y,x,z\right)} \right| \le
N\left(f;y,x,z\right)\cdot \left\| {f^{\left( n \right)} }
\right\|_p,\label{eq3.18}
\end{align}
$\forall p\in \left[1,\infty\right]$ and all $a\le y\le x \le z
\le b$, where
\begin{align}
\label{eq3.19}N\left(f;x,a,b\right):=\frac{1}{{n }} \left\{
\begin{array}{l}
 \mathop {\sup }\limits_{a \le t \le b} \left\{ {\left| {Q_{n - 1}
\left( t \right)} \right| \left| {K\left(
{t;y,x,z} \right)} \right|} \right\},\,\,\,\,\,\,\,\,\,\,\,\,\,\,\,\,\,\,\,\,\,\,p = 1 \\
 \\
  \left( {\int_a^b {\left| {Q_{n - 1}
\left( t \right)} \right|^{q} \left| {K\left(
{t;y,x,z} \right)} \right|^q dt} } \right)^{1/q} ,\,\,\,\,\,\,\,\,\,\,1 < p < \infty  \\
 \\
   \int_a^b {\left| {Q_{n - 1}
\left( t \right)} \right| \left| {K\left(
{t;y,x,z} \right)} \right|dt} ,\,\,\,\,\,\,\,\,\,\,\,\,\,\,\,\,\,\,\,\,\,\,\,\,\,\,\,\,\,\,\,\,p = \infty  \\
 \end{array} \right..
\end{align}
\end{theorem}

\begin{proof}
Utilizing the triangle integral inequality on the identity
\eqref{eq3.17} and employing some known norm inequalities we get
\begin{align*}
\left|{\mathcal{E}_n\left(f,Q_n;y,x,z\right)} \right| &=\left|{
\frac{{\left( { - 1} \right)^{n - 1} }}{n}\int_a^b {Q_{n - 1}
\left( t \right)K\left( {t;y,x,z} \right)f^{\left( n \right)}
\left( t \right)dt}} \right|
\\
&\le \frac{1}{{n }}\int_a^b {\left| {Q_{n - 1} \left( t \right)}
\right| \left|{K\left( {t;y,x,z} \right)}\right| \left|{f^{\left(
n \right)} \left( t \right)}\right|dt}
\\
&\le  \frac{1}{{n }} \left\{ \begin{array}{l}
 \left\| {f^{\left( n \right)} } \right\|_1 \mathop {\sup }\limits_{a \le t \le b} \left\{ {\left| {Q_{n - 1}
\left( t \right)} \right| \left| {K\left( {t;y,x,z} \right)} \right|} \right\},\,\,\,\,\,\,\,\,\,\,\,\,\,\,\,\,\,\,\,\,\,\,\,\,\,\,p = 1 \\
 \\
 \left\| {f^{\left( n \right)} } \right\|_p \left( {\int_a^b {\left| {Q_{n - 1}
\left( t \right)} \right|^{q} \left| {K\left( {t;y,x,z} \right)} \right|^q dt} } \right)^{1/q} ,\,\,\,\,\,\,\,\,\,\,1 < p < \infty  \\
 \\
 \left\| {f^{\left( n \right)} } \right\|_\infty  \int_a^b {\left| {Q_{n - 1}
\left( t \right)} \right| \left| {K\left( {t;y,x,z} \right)} \right|dt} ,\,\,\,\,\,\,\,\,\,\,\,\,\,\,\,\,\,\,\,\,\,\,\,\,\,\,\,\,\,\,\,\,\,p = \infty  \\
 \end{array} \right.\\
\\
&=N\left(f;x,a,b\right)  \left\| {f^{\left( n \right)} }
\right\|_p, \qquad \forall p\in \left[1,\infty\right],
\end{align*}
and this completes the proof.
\end{proof}

\begin{corollary}
\label{cor2}Under the assumptions of Theorem \ref{thm8}, we have
\begin{multline}
\left|{\mathcal{E}_n\left(f,Q_n;y,x,z\right)} \right|
\\
\le  \frac{1}{n}\max \left\{ {\left( {y  - a} \right),\left(
{\frac{{z - y }}{2} + \left| {x - \frac{{y  +z }}{2}} \right|}
\right),\left( {b - z } \right)} \right\}\cdot\left\| {Q_{n-1}}
\right\|_q\cdot \left\| {f^{\left( n \right)} }
\right\|_p,\label{eq3.20}
\end{multline}
$\forall p,q \ge 1$ with $\frac{1}{p}+\frac{1}{q}=1$ and all $a\le
y\le x \le z \le b$, where
\begin{align*}
\left\| {Q_{n-1}} \right\|_q=\left\{
\begin{array}{l}
 \mathop {\sup }\limits_{a \le t \le b} \left\{ {\left| {Q_{n - 1}
\left( t \right)} \right|} \right\},\,\,\,\,\,\,\,\,\,\,\,\,\,\,\,\,\,\,\,\,\,\,\,\,\,\,q = \infty \\
 \\
  \left( {\int_a^b {\left| {Q_{n - 1}
\left( t \right)} \right|^{q}dt} } \right)^{1/q} ,\,\,\,\,\,\,\,\,\,\,1 < q < \infty  \\
 \\
   \int_a^b {\left| {Q_{n - 1}
\left( t \right)} \right| dt} ,\,\,\,\,\,\,\,\,\,\,\,\,\,\,\,\,\,\,\,\,\,\,\,\,\,\,\,\,\,\,\,\,\,q = 1  \\
 \end{array} \right.
\end{align*}
\end{corollary}

\begin{proof}
In \eqref{eq3.20}, it is easy to verify that
\begin{align*}
N\left(f;x,y,z\right)&\le \frac{1}{n}\mathop {\sup }\limits_{a \le
t \le b} \left\{ { \left| {K\left( {t;y,x,z} \right)} \right|}
\right\}\cdot \left\| {Q_{n-1}} \right\|_q
\\
&= \frac{1}{n}\max \left\{ {\left( {y  - a} \right),\left(
{\frac{{z - y }}{2} + \left| {x - \frac{{y  +z }}{2}} \right|}
\right),\left( {b - z } \right)} \right\}\cdot\left\| {Q_{n-1}}
\right\|_q,
\end{align*}
$\forall q\in \left[1,\infty\right]$ and all $a\le y\le x \le z
\le b$.
\end{proof}

In particular case we have
\begin{corollary}
 Under the assumptions of Theorem \ref{thm8}, we have
\begin{multline}
\left|{\frac{1}{n}\left[ {\frac{\left( {x - a} \right)f\left( y
\right) + \left( {b - x} \right)f\left( z \right)}{b-a} +
\sum\limits_{k = 1}^{n - 1} {
\left\{{\widetilde{T}_k\left(y,x,z\right)+\widetilde{F}_k \left(
{a,b} \right) }\right\}} } \right] - \frac{1}{b-a}\int_a^b
{f\left( y \right)dy}} \right|
\\
\le  \frac{1}{n\left(b-a\right)}\max \left\{ {\left( {y  - a}
\right),\left( {\frac{{z - y }}{2} + \left| {x - \frac{{y  +z
}}{2}} \right|} \right),\left( {b - z } \right)}
\right\}\cdot\left\| {S_{n-1}} \right\|_q\cdot \left\| {f^{\left(
n \right)} } \right\|_p,\label{eq3.21}
\end{multline}
$\forall p,q \ge 1$ with $\frac{1}{p}+\frac{1}{q}=1$ and all $a\le
y\le x \le z \le b$, where
\begin{align*}
\left\| {S_{n-1}} \right\|_q=\left\{
\begin{array}{l}
\frac{1}{{n!}}\left[ {\max \left\{ {\left( {y - a} \right),\left(
{\frac{{z - y}}{2} + \left| {\frac{{y + z}}{2} - x} \right|}
\right),\left( {b - z}
\right)} \right\}} \right]^{n}, \,\,\,\,\,\,\,\,\,q = \infty \\
 \\
 \frac{\left( {y - a} \right)^{\left(n-1\right)q + 1}  + \left(
{x - y} \right)^{\left(n-1\right)q + 1}  + \left( {z - x}
\right)^{\left(n-1\right)q + 1}  + \left( {b - z}
\right)^{\left(n-1\right)q + 1} }{{\left( {\left(n-1\right)q + 1}
\right) \left(n-1\right)!}},\,\,\,\,\,\,\,\,\,\,\,1 < q < \infty  \\
 \\
  \frac{\left( {y - a} \right)^{n}  + \left(
{x - y} \right)^{n}  + \left( {z - x} \right)^{n}  + \left( {b -
z} \right)^{n} }{n!},\qquad\qquad\qquad\,\,\,\,\,\,\,\,\,\,\,\,\,\,\,\,\,\,\,\,\,\,\,\,\,\,\,\,\,\,\,\,\,\,\,\,\,q = 1  \\
 \end{array} \right.
\end{align*}
\end{corollary}

\begin{proof}
The proof is an immediate consequence of Corollary \ref{cor2}, by
setting $Q_{k}\left( {t} \right)=\frac{1}{k!}\left( {t - a}
\right)^{k}$.
\end{proof}

\begin{remark}
In all above estimates, if one assumes that $f^{(n)}$ is convex,
$r$-convex, quasi-convex, $s$-convex, $P$-convex, or $Q$-convex;
we can obtain other new bounds involving convexity.
\end{remark}

%===========================================================================================================
\subsection{ Fink representation for Guessab--Schmeisser formula}
%===========================================================================================================
In this part, we give some special formulas of the previous
expansion via Fink approach with some error estimates.

Seeking Taylor like expansion of \eqref{eq3.1}, we set $Q_k \left(
t \right)=\frac{\left( t-\alpha \right)^k}{k!}$, $a\le \alpha \le
b$, then we have the expansion
\begin{multline}
\frac{1}{n}\left[ {\frac{\left( {x - a} \right)f\left( y \right) +
        \left( {b - x} \right)f\left( z \right)}{b-a} + \sum\limits_{k =
        1}^{n - 1} {
        \left\{{\widetilde{T}_k\left(y,x,z\right)+\widetilde{F}_k \left(
            {\alpha;a,b} \right) }\right\}} } \right] - \frac{1}{b-a}\int_a^b
{f\left(s \right)ds}
\\
= \frac{1}{n!\left(b-a \right)}\int_a^b {\left( \alpha-t
    \right)^{n-1}K\left( {t;y,x,z} \right)f^{\left( n \right)} \left(
    t \right)dt},\label{eq3.22}
\end{multline}
where
\begin{align*}
\widetilde{T}_k\left(y,x,z\right):=\frac{1}{\left( { b-a}
    \right)k!} \left\{ {\left( {x - a} \right)\left( \alpha-y
    \right)^k f^{\left( k \right)} \left( y \right)+ \left( { - 1}
    \right)^k \left( {b - x} \right)\left( z-\alpha \right)^kf^{\left(
        k \right)} \left( z \right)} \right\}.
\end{align*}
and
\begin{align*}
\widetilde{F}_k \left( {\alpha;a,b} \right) = \frac{\left( {n - k}
    \right)}{\left( {b - a} \right)k!} \left[ { \left( \alpha-a
    \right)^kf^{\left( {k - 1} \right)} \left( a \right) + \left( { -
        1} \right)^{k+1}  \left( b-\alpha \right)^kf^{\left( {k - 1}
        \right)} \left( b \right)} \right],
\end{align*}
for all $a\le y \le x \le z \le b$, which gives   Fink
representation of  general two-point Ostrowski's formula. One
could get more informative representation by choosing $\alpha=x$
in \eqref{eq3.22}

\begin{remark}
By setting $x=y=z=\alpha$ in \eqref{eq3.22}, we refer to Fink
representation \eqref{eq1.6}.
\end{remark}

 Furthermore, the Fink representation for Guessab--Schmeisser
formula is deduced by setting $y=h$, $z=a+b-h$ and
$x=\frac{a+b}{2}$ in \eqref{eq3.1}, so we get:
\begin{multline}
\frac{1}{n}\left[ {\frac{ f\left( h \right) + f\left( a+b-h
        \right)}{2} + \sum\limits_{k = 1}^{n - 1} {
        \left\{{T_k\left(h,\frac{a+b}{2},a+b-h\right)+F_k \left( {a,b}
            \right) }\right\}} } \right] - \frac{1}{b-a}\int_a^b {f\left( s
    \right)ds}
\\
= \frac{{\left( { - 1} \right)^{n - 1} }}{n\left(b-a
    \right)}\int_a^b {Q_{n - 1} \left( t \right)K\left(
    {t;h,\frac{a+b}{2},a+b-h} \right)f^{\left( n \right)} \left( t
    \right)dt},\label{eq3.23}
\end{multline}
where
\begin{align*}
T_k\left(h,\frac{a+b}{2},a+b-h\right):=\frac{\left( { - 1}
    \right)^k}{2} \left\{ { Q_k \left( h \right)f^{\left( k \right)}
    \left( h \right)+
    Q_k \left( a+b-h \right)f^{\left( k \right)}
    \left( a+b-h \right)} \right\}.
\end{align*}
for all $a\le h \le \frac{a+b}{2}$. In particular, for $Q_k \left(
a+b-t \right)=\left(-1\right)^kQ_k \left( t \right)$ $a\le t  \le
 \frac{a+b}{2}$, we have
\begin{multline*}
\frac{1}{n}\left[ {\frac{ f\left( h \right) + f\left( a+b-h
        \right)}{2} + \sum\limits_{k = 1}^{n - 1} {
        \left\{{T_k\left(h,\frac{a+b}{2},a+b-h\right)+F_k \left( {a,b}
            \right) }\right\}} } \right] - \frac{1}{b-a}\int_a^b {f\left( s
    \right)ds}
\\
= \frac{{\left( { - 1} \right)^{n - 1} }}{n\left(b-a
    \right)}\int_a^b {Q_{n - 1} \left( t \right)K\left(
    {t;h,\frac{a+b}{2},a+b-h} \right)f^{\left( n \right)} \left( t
    \right)dt},
\end{multline*}
where
\begin{align*}
T_k\left(h,\frac{a+b}{2},a+b-h\right):=\frac{\left( { - 1}
    \right)^k}{2}  Q_k \left( h \right)\left\{ {f^{\left( k \right)}
    \left( h \right)+
  \left( -1 \right)^kf^{\left( k \right)}
    \left( a+b-h \right)} \right\}.,
\end{align*}
for all $a\le h \le \frac{a+b}{2}$.

Now, by substituting
$Q_k\left(t\right)=\frac{\left(t-h\right)^k}{k!}$  in
\eqref{eq3.23}, so that we get
\begin{multline}
\frac{1}{n}\left[ {\frac{{f\left( h\right) + f\left( {a + b - h}
            \right)}}{2} + \sum\limits_{k = 1}^{n - 1} { \widetilde{F_k }
        \left(a,b\right)} } \right] - \frac{1}{{b - a}}\int_a^b {f\left( y
    \right)dy}
\\
= \frac{1}{{n!\left({b-a}\right) }}\int_a^b {\left( {h - t}
    \right)^{n - 1} S\left( {t,h} \right)f^{\left( n \right)} \left( t
    \right)dt}. \label{eq3.24}
\end{multline}
 Since $Q_k \left( b
\right) = \left( { - 1} \right)^k Q_k \left( {a} \right)$, then
\begin{align*}
\widetilde{F_k }\left( {a,b} \right) &= \frac{{\left( { - 1}
        \right)^k \left( {n - k} \right)}}{{b - a}}Q_k \left(a
\right)\left[ {f^{\left( {k - 1} \right)} \left( a \right) -\left(
    { - 1} \right)^k f^{\left( {k - 1} \right)} \left( b \right)}
\right]
\\
&=\frac{{\left( { - 1} \right)^k \left( {n - k} \right)}}{{b -
        a}}\frac{\left(a-h\right)^k}{k!}\left[ {f^{\left( {k - 1} \right)}
    \left( a \right) - \left( { - 1} \right)^k f^{\left( {k - 1}
        \right)} \left( b \right)} \right]
\\
&=\frac{{ \left( {n - k} \right)}}{{b -
        a}}\frac{\left(h-a\right)^k}{k!}\left[ {f^{\left( {k - 1} \right)}
    \left( a \right) + \left( { - 1} \right)^{k+1} f^{\left( {k - 1}
        \right)} \left( b \right)} \right].
\end{align*}
Also, we note that
\begin{align*}
Q_k\left(\frac{a+b}{2}\right)=\frac{\left(\frac{a+b}{2}-h\right)^k}{k!}
&=\left({-1}\right)^k\frac{\left(h-\frac{a+b}{2}\right)^k}{k!}
\\
&=\left({-1}\right)^k Q_k\left(\frac{a+b}{2}\right)=
\left({-1}\right)^kQ_k\left(a+b-\frac{a+b}{2}\right),
\end{align*}
this gives that
\begin{align*}
0=\frac{{ \left( {n - k} \right)}}{{\left(b -
a\right)k!}}\left[Q_k\left(\frac{a+b}{2}\right)-\left({-1}\right)^k
Q_k\left(\frac{a+b}{2}\right)\right]&=\frac{{ \left( {n - k}
\right)}}{{\left(b - a\right)k!}}\cdot
\left({1+\left({-1}\right)^{k+1}}\right)
Q_k\left(\frac{a+b}{2}\right).
\end{align*}
By our choice of $Q_k$; we have $Q_k\left(\frac{a+b}{2}\right)=
\frac{\left(\frac{a+b}{2}-h\right)^k}{k!}$, for all $x\in \left[
{a,{\textstyle{{a + b} \over 2}}} \right]$, therefore we can write
\begin{align*}
\widetilde{F_k }\left( {a,b} \right)+0 &=\widetilde{F_k }\left(
{a,b} \right)+\frac{{ \left( {n - k} \right)}}{{\left(b -
a\right)k!}}\left({1+\left({-1}\right)^{k+1}}\right)
Q_k\left(\frac{a+b}{2}\right)
\\
&= \frac{{ \left( {n - k} \right)}}{{\left(b - a\right)k!}}
\left[{\left(h-a\right)^k\left( {f^{\left( {k - 1} \right)} \left(
        a \right) + \left( { - 1} \right)^{k+1} f^{\left( {k - 1} \right)}
        \left( b \right)} \right) }\right.
\\
&\qquad\qquad\left.{+ \left({1+\left({-1}\right)^{k+1}}\right)
    \left(\frac{a+b}{2}-h\right)^k}\right]
\\
&=: G_k\left(h;a,b\right).
\end{align*}
for all $h\in \left[a,\frac{a+b}{2}\right]$. Hence, we just proved
that
\begin{corollary}
Let $I$ be a real interval, $a,b \in I^{\circ}$ $(a<b)$. Let $f:I
\to \mathbb{R}$ be such that $f^{\left( n \right)}$ is absolutely
continuous on $I$ for $n\ge1$ with $ \left( {\cdot-t }
\right)^{n-1}S\left(\cdot,t\right)f^{\left( n \right)} \left( t
\right)$ is integrable. Then we have the representation
\begin{multline}
\frac{1}{n}\left[ {\frac{{f\left( x\right) + f\left( {a + b - x}
            \right)}}{2} + \sum\limits_{k = 1}^{n - 1} {G
        \left(x;a,b\right)} } \right] - \frac{1}{{b - a}}\int_a^b {f\left( y
    \right)dy}
\\
= \frac{1}{{n!\left({b-a}\right) }}\int_a^b {\left( {x - t}
    \right)^{n - 1} S\left( {t,x} \right)f^{\left( n \right)} \left( t
    \right)dt}. \label{eq3.25}
\end{multline}
for all $x\in \left[a,\frac{a+b}{2}\right]$, where
\begin{multline}
G_k\left(x\right):=G_k\left(x;a,b\right)  = \frac{{\left( {n - k}
        \right)}}{{k!\left( {b - a} \right)}} \cdot \left\{{ \left( {x -
        a} \right)^k \left[ {f^{\left( {k - 1} \right)} \left( a \right) +
        \left( { - 1} \right)^{k + 1} f^{\left( {k - 1} \right)} \left( b
        \right)} \right] }\right.
\\
\left.{+ \left( {1 + \left( { - 1} \right)^{k + 1} } \right)\left(
    {\frac{{a + b}}{2} - x} \right)^k f^{\left( {k - 1} \right)}
    \left( {\frac{{a + b}}{2}} \right) }\right\}, \label{eq3.26}
\end{multline}
for all $x\in \left[a,\frac{a+b}{2}\right]$, and
\begin{align}
S\left( {t,h} \right) &= \left\{ \begin{array}{l}
t - a,\,\,\,\,\,\,\,\,\,\,\,\,\,\,\,\,\,\, t \in \left[ {a,x} \right] \\
t - \frac{{a + b}}{2},\,\,\,\,\,\,\,\,\,\,\,t \in \left( {x,a + b - x} \right) \\
t - b,\,\,\,\,\,\,\,\,\,\,\,\,\,\,\,\,\,\, t \in \left[ {a + b - x,b} \right] \\
\end{array} \right.\label{eq3.27}
\\
&:=T_k\left(x,\frac{a+b}{2},a+b-x\right).\nonumber
\end{align}

\end{corollary}

\begin{theorem}
    \label{thm10}Under the assumptions of Theorem \ref{thm5}. We have
    \begin{align}
    \left|{ \frac{1}{n}\left( {\frac{{f\left( x \right) + f\left( {a +
                        b - x} \right)}}{2} + \sum\limits_{k = 1}^{n - 1} {G_k\left(x\right) } } \right)
        - \frac{1}{{b - a}}\int_a^b {f\left( y \right)dy} }\right|
        \le K\left( {n,p,x} \right)\left\| {f^{\left( n \right)} }
    \right\|_p\label{eq3.28}
    \end{align}
    holds for all $x\in \left[a,\frac{a+b}{2}\right]$,  where
    \begin{align}
    K\left( {n,p,x} \right) =   \left\{ \begin{array}{l}
    \frac{1}{n\cdot n!\left( {b-a} \right)}\left( {\frac{{n - 1}}{n}}
    \right)^{n - 1} \left[ {\frac{{b - a}}{4} + \left| {x - \frac{{3a
                    + b}}{4}} \right|} \right]^n
    ,\,\,\,\,\,\,\,\,\,\,\,\,\,\,\,\,\,\,\,\,\,\,\,\,\,\,\,{\rm{if}}\,\,p = 1 \\
    \\
    \frac{{2^{1/q} }}{{n!\left( {b - a} \right)}}\left[ {\left( {x -
            a} \right)^{nq + 1}  + \left( {\frac{{a + b}}{2} - x} \right)^{nq
            + 1} } \right]^{1/q}
    \\
 \times {\rm{B}}^{\frac{1}{q}} \left( {\left( {n - 1}
        \right)q + 1,q + 1} \right),
    \,\,\,\,\,\,\,\,\,\,\,\,\,\,\,\,\,\,\,\,{\rm{if}}\,\,  1 < p \le \infty, \,q=\frac{p}{p-1}  \\
    \end{array} \right..
\label{eq3.29}
    \end{align}
    The constant $K\left( {n,p,x} \right)$ is the best possible in the
    sense that it cannot be replaced by a smaller ones.
\end{theorem}

\begin{proof}
    Utilizing the triangle integral inequality on the identity
    \eqref{eq3.23} and employing some known norm inequalities we get
    \begin{align*}
    &\left|{ \frac{1}{n}\left( {\frac{{f\left( x \right) + f\left( {a
                        + b - x} \right)}}{2} + \sum\limits_{k = 1}^{n - 1}
            {G_k\left(x\right) } } \right) - \frac{1}{{b - a}}\int_a^b
        {f\left( y \right)dy} } \right|
    \\
    &\le \frac{1}{{n!\left({b-a}\right) }}\int_a^b {\left| {x - t}
        \right|^{n - 1} \left|{S\left( {t,x} \right)}\right|
        \left|{f^{\left( n \right)} \left( t \right)}\right|dt}
    \\
    &\le \left\{ \begin{array}{l}
    \left\| {f^{\left( n \right)} } \right\|_1 \mathop {\sup }\limits_{a \le t \le b} \left\{ {\left| {x - t} \right|^{n - 1} \left| {k\left( {t,x} \right)} \right|} \right\},\,\,\,\,\,\,\,\,\,\,\,\,\,\,\,\,\,\,\,\,\,\,\,\,\,\,p = 1 \\
    \\
    \left\| {f^{\left( n \right)} } \right\|_p \left( {\int_a^b {\left| {x - t} \right|^{\left( {n - 1} \right)q} \left| {k\left( {t,x} \right)} \right|^q dt} } \right)^{1/q} ,\,\,\,\,\,\,\,\,\,\,1 < p < \infty  \\
    \\
    \left\| {f^{\left( n \right)} } \right\|_\infty  \int_a^b {\left| {x - t} \right|^{n - 1} \left| {k\left( {t,x} \right)} \right|dt} ,\,\,\,\,\,\,\,\,\,\,\,\,\,\,\,\,\,\,\,\,\,\,\,\,\,\,\,\,\,\,\,\,\,p = \infty  \\
    \end{array} \right..
    \end{align*}
    It is easy to find that for $p=1$, we have
    \begin{align*}
    \mathop {\sup }\limits_{a \le t \le b} \left\{ {\left| {x - t}
        \right|^{n - 1} \left| {S\left( {t,x} \right)} \right|}
    \right\}&=\frac{1}{n}\left( {\frac{{n - 1}}{n}} \right)^{n - 1}
    \max \left\{ {\left( {x - a} \right)^n ,\left( {\frac{{a + b}}{2}
            - x} \right)^n } \right\}
    \\
    &= \frac{1}{n}\left( {\frac{{n - 1}}{n}} \right)^{n - 1} \left[
    {\frac{{b - a}}{4} + \left| {x - \frac{{3a + b}}{4}} \right|}
    \right]^n,
    \end{align*}
    and for $1<p<\infty$, we have
    \begin{align*}
    &\int_a^b {\left| {x - t} \right|^{\left( {n - 1} \right)q} \left|
        {S\left( {t,x} \right)} \right|^q dt}
    \\
    &= \int_a^x {\left| {x - t} \right|^{\left( {n - 1} \right)q}
        \left( {t-a} \right)^q dt} + \int_x^{a+b-x} {\left| {x - t}
        \right|^{\left( {n - 1} \right)q} \left| {t-\frac{a+b}{2}}
        \right|^q dt}
    \\
    &\qquad+ \int_{a+b-x}^b {\left| {x - t} \right|^{\left( {n - 1}
            \right)q} \left( {b-t} \right)^q dt}
    \\
    &= 2  \left[ {\left( {x - a} \right)^{nq + 1}  + \left( {\frac{{a
                    + b}}{2} - x} \right)^{nq + 1} } \right]  \left(
    {\int_{\rm{0}}^{\rm{1}} {\left( {1 - s} \right)^{\left( {n - 1}
                \right)q} s^q ds} } \right)
    \\
    &= 2 \left[ {\left( {x - a} \right)^{nq + 1}  + \left( {\frac{{a +
                    b}}{2} - x} \right)^{nq + 1} } \right] {\rm{B}} \left( {\left( {n
            - 1} \right)q + 1,q + 1} \right)
    \end{align*}
    where, we use the substitutions $t=\left({1-s}\right)a+s x$,
    $t=\left({1-s}\right)x+s\left({a+b-x}\right)$ and
    $t=\left({1-s}\right)\left({a+b-x}\right)+sb$; respectively.The
    third case, $p=\infty$ holds by setting $p= \infty$ and $q=1$,
    i.e.,
    \begin{align*}
    &\int_a^b {\left| {x - t} \right|^{\left( {n - 1} \right)} \left|
        {S\left( {t,x} \right)} \right| dt}= 2\left[ {\left( {x - a}
        \right)^{n + 1}  + \left( {\frac{{a + b}}{2} - x} \right)^{n + 1}
    } \right]{\rm{B}}\left( {n,2} \right),
    \end{align*}
    where  ${\rm{B}}\left( {\cdot,\cdot} \right)$ is the Euler beta
    function.  To argue the sharpness, we consider first when $1<p\le
    \infty$, so that the equality in \eqref{eq3.28} holds when
    \begin{align*}
    f^{\left( n \right)} \left( t \right) = \left| {x - t}
    \right|^{\left( {n - 1} \right)q - 1} \left| {S\left( {t,x}
        \right)} \right|^{q - 1} {\mathop{\rm sgn}} \left\{ {\left( {x -
            t} \right)^{n - 1} S\left( {t,x} \right)} \right\},
    \end{align*}
    thus the inequality \eqref{eq3.28} holds for $1<p\le \infty$. In
    case that $p=1$, setting
    \begin{align*}
    g\left( {t,x} \right) = \left( {x - t} \right)^{n - 1} S\left(
    {t,x} \right) \qquad \forall x\in \left[a,\textstyle{{a + b} \over
        2}\right],
    \end{align*}
    let $t_0$ be the point that gives the supremum. If
    $t_0=\frac{x+\left(n-1\right)a}{n}$, we take
    \begin{align*}
    f_{\varepsilon}^{\left( n \right)} \left( t \right) = \left\{
    \begin{array}{l}
    \varepsilon ^{ - 1} ,\,\,\,\,\,\,t \in \left( {t_0  - \varepsilon ,t_0 } \right) \\
    0,\,\,\,\,\,\,\,\,\,\,{\rm{otherwise}} \\
    \end{array} \right..
    \end{align*}
    Since
    \begin{align*}
    \left| {\int_a^b {g\left( {t,x} \right)f^{\left( n \right)} \left(
            t \right)dt} } \right| = \frac{1}{\varepsilon }\left| {\int_{t_0 -
            \varepsilon }^{t_0 } {g\left( {t,x} \right)dt} } \right| &\le
    \frac{1}{\varepsilon }\int_{t_0  - \varepsilon }^{t_0 } {\left|
        {g\left( {t,x} \right)} \right|dt}
    \\
    &\le \mathop {\sup }\limits_{t_0  - \varepsilon  \le t \le t_0 }
    \left| {g\left( {t,x} \right)} \right| \cdot \frac{1}{\varepsilon
    }\int_{t_0  - \varepsilon }^{t_0 } {dt}
    \\
    &= \left| {g\left( {t_0 ,x} \right)} \right|,
    \end{align*}
    also, we have
    \begin{align*}
    \mathop {\lim }\limits_{\varepsilon  \to 0^ +  }
    \frac{1}{\varepsilon }\int_{t_0  - \varepsilon }^{t_0 } {\left|
        {g\left( {t,x} \right)} \right|dt}  = \left| {g\left( {t_0 ,x}
        \right)} \right|=C\left( n,1,x \right)
    \end{align*}
    proving that $C\left( n,1,x \right)$ is the best possible.
\end{proof}

\begin{remark}
    A representation of Cerone-Dragomir formula \cite{Cerone1} (see also \cite{Cerone2}) via Fink
    approach can be also deduced from \eqref{eq3.1} by setting $y=a$,
    $z=b$ and $x \in \left[a, b\right]$. Taylor type expansion can be
    also given using  the representation \eqref{eq3.15}.
\end{remark}

%=============================================================================
\section{Error bounds via Chebyshev--Gr\"{u}ss type inequalities}
\label{sec4}
%=============================================================================
In this section, we highlight the role of \v{C}eby\v{s}ev
functional in integral approximations by using the
\v{C}eby\v{s}ev--Gr\"{u}ss type inequalities \eqref{eq3.25}.

The famous \v{C}eby\v{s}ev functional
\begin{align}
\label{eq4.1}\mathcal{T}\left( {h_1,h_1} \right) = \frac{1}{{d -
		c}}\int_c^d {h_1\left( t \right)h_2\left( t \right)dt}  -
\frac{1}{{d - c}}\int_c^d {h_1\left( t \right)dt} \cdot
\frac{1}{{d - c}}\int_c^d {h_2\left( t \right)dt}.
\end{align}
has multiple applications in several area of mathematical sciences
specially in Integral Approximations of real functions. For more
detailed history see \cite{MPF}.

It is well known that the pre--Gr\"{u}ss inequality reads:
\begin{align}
\label{eq4.2}\left|{\mathcal{T}\left( {h_1,h_2} \right)} \right|
\le \sqrt{\mathcal{T}\left( {h_1,h_1}
	\right)}\sqrt{\mathcal{T}\left( {h_2,h_2} \right)},
\end{align}
for all measurable functions $h_1,h_2$ defined on $[a,b]$. This
inequality was used by Gr\"{u}ss to prove the second inequality in
\eqref{eq4.1}. A ramified inequality could be deduced as follows:
\begin{align}
\label{eq4.3}\left|{\mathcal{T}\left( {h_1,h_2} \right)} \right|
\le \frac{1}{2}\left(\Phi-\phi\right)\sqrt{\mathcal{T}\left(
	{h_2,h_2} \right)},
\end{align}
where $h_1,h_2:[a,b]\to \mathbb{R}$ are assumed to be such that
$h_2$ is integrable and $h_2$ is measurable bounded on $[a,b]$,
i.e., there exist constants $\phi, \Phi$ such that $\phi\le
h_2(t)\le \Phi$, for $t\in [a,b]$.\\

The most famous bounds of the \v{C}eby\v{s}ev functional are
incorporated in the following theorem:
\begin{theorem}\label{thm11}
	Let $f,g:[c,d] \to \mathbb{R}$ be two absolutely continuous
	functions, then
	\begin{multline}
	\left|{\mathcal{T}\left( {h_1,h_2} \right)} \right| \\\le\left\{
	\begin{array}{l} \frac{{\left( {d - c} \right)^2 }}{{12}}\left\|
	{h^{\prime}_1} \right\|_\infty  \left\| {h^{\prime}_2}
	\right\|_\infty
	,\,\,\,\,\,\,\,\,\,{\rm{if}}\,\,h^{\prime}_1,h^{\prime}_2 \in L_{\infty}\left(\left[c,d\right]\right),\,\,\,\,\,\,\,\,{\rm{proved \,\,in \,\,}}{\text{\cite{Cebysev}}}\\
	\\
	\frac{1}{4}\left( {M_1 - m_1} \right)\left( {M_2 - m_2}
	\right),\,\,\, {\rm{if}}\,\, m_1\le h_1 \le M_1,\,\,\,m_2\le h_2
	\le
	M_2, \,\,{\rm{proved \,\,in \,\,}}{\text{\cite{Gruss}}}\\
	\\
	\frac{{\left( {d - c} \right)}}{{\pi ^2 }}\left\| {h^{\prime}_1}
	\right\|_2 \left\| {h^{\prime}_2} \right\|_2
	,\,\,\,\,\,\,\,\,\,\,\,\,\,\,\,\,{\rm{if}}\,\,h^{\prime}_1,h^{\prime}_2
	\in
	L_{2}\left(\left[c,d\right]\right),\,\,\,\,\,\,\,\,\,\,\,{\rm{proved
			\,\,in\,\,
	}}{\text{\cite{L}}}\\
	\\
	\frac{1}{8}\left( {d - c} \right)\left( {M - m} \right) \left\|
	{h^{\prime}_2} \right\|_{\infty},\,\,\, {\rm{if}}\,\, m\le h_1 \le
	M,\,h^{\prime}_2 \in L_{\infty}\left(\left[c,d\right]\right),
	\,\,{\rm{proved \,\,in \,\,}}{\text{\cite{O}}}
	\end{array} \right. \label{eq4.4}
	\end{multline}
	The constants $\frac{1}{12}$, $\frac{1}{4}$, $\frac{1}{\pi^2}$ and
	$\frac{1}{8}$ are the best possible.
\end{theorem}

Setting $h_1\left(t\right)=\frac{1}{n!}f^{(n)}\left(t\right)$ and
$h_2\left(t\right)=\left(x-t\right)^{n-1}S\left(t,x\right)$,  we
have
\begin{align*}
\mathcal{C}\left( {h_1,h_2} \right) &= \frac{1}{n!\left(b-a\right)
} \int_a^b {\left( {x - t} \right)^{n - 1} S\left( {t,x}
	\right)f^{\left( n \right)} \left( t \right)dt}
\\
&\qquad-  \frac{1}{n!} \cdot\frac{1}{b-a} \int_a^b {\left( {x - t}
	\right)^{n - 1}  S\left( {t,x} \right) dt} \cdot\frac{1}{b-a}
\int_a^b { f^{\left( n \right)} \left( t \right)dt}
\\
&= \frac{1}{n!\left(b-a\right) } \int_a^b {\left( {x - t}
	\right)^{n - 1} S\left( {t,x} \right)f^{\left( n \right)} \left( t
	\right)dt}
\\
&\qquad- \frac{{2 }}{{n!\left( {b - a} \right)}}\left[ {\left( {x
		- a} \right)^{n + 1}  + \left( {\frac{{a + b}}{2} - x} \right)^{n
		+ 1} } \right] {\rm{B}} \left( {n,2} \right)
\\
&\qquad\qquad\times \frac{{f^{\left( {n - 1} \right)} \left( b
		\right) - f^{\left( {n - 1} \right)} \left( a \right)}}{{b - a}}
\end{align*}
which means
\begin{align}
\mathcal{C}\left( {h_1,h_2} \right) &= \frac{1}{n}\left(
{\frac{{f\left( x \right) + f\left( {a + b - x} \right)}}{2} +
	\sum\limits_{k = 1}^{n - 1} {G_k\left(x\right)} } \right) -
\frac{1}{{b - a}}\int_a^b {f\left( y \right)dy}
\nonumber\\
&\qquad- \frac{{2 }}{{\left( {n + 1} \right)!n\left( {b - a}
		\right)}}\left[ {\left( {x - a} \right)^{n + 1}  + \left(
	{\frac{{a + b}}{2} - x} \right)^{n + 1} } \right]
\nonumber\\
&\qquad\qquad\times \frac{{f^{\left( {n - 1} \right)} \left( b
		\right) - f^{\left( {n - 1} \right)} \left( a \right)}}{{b - a}}
\nonumber\\
&:=\mathcal{P}\left( {f;x,n} \right).
\end{align}
for all $x\in \left[a,\frac{a+b}{2}\right]$.

\begin{theorem}
	\label{thm14}Let $I$ be a real interval, $a,b \in I^{\circ}$
	$(a<b)$. Let $f:I \to \mathbb{R}$ be $(n+1)$-times differentiable
	on $I^{\circ}$ such that $f^{(n+1)}$ is absolutely continuous on
	$I^{\circ}$ with $\left( {\cdot - t} \right)^{n - 1} k\left(
	{t,\cdot} \right)f^{\left( n \right)} \left( t \right)$ is
	integrable. Then, for all $n\ge2$ we have
	\begin{multline}
	\label{eq5.2}   \left|{\mathcal{P}\left( {f;x,n} \right)} \right|
	\\
	\le\left\{
	\begin{array}{l}
	\left( {b-a} \right)^{2}\left( {\frac{{n - 2}}{n}} \right)^{n - 2}
	\frac{{n^2  - 2n + 2}}{{12n\cdot (n!)^2}} \left[ {\frac{{b -
				a}}{4} + \left| {x - \frac{{3a + b}}{4}} \right|}
	\right]^{n-1}\cdot \left\|{f^{\left( {n + 1} \right)}
	}\right\|_{\infty} ,\,\,\,\,\,\,\,\,\,{\rm{if}}\,\,f^{(n+1)} \in
	L_{\infty}\left(\left[a,b\right]\right)\\
	\\
	\left( {\frac{{n - 2}}{n}} \right)^{n - 2} \frac{{n^2  - 2n +
			2}}{{4n\cdot (n!)^2}}   \left( {2^{ - n - 2}  - 2^{ - 2n - 2} }
	\right) \left( {b -  a} \right)^{n-2}\cdot  \left( {M -  m}
	\right),\,\,\,
	{\rm{if}}\,\, m\le f^{(n)} \le M, \\
	\\
	\frac{{b-a}}{{(n!)^2\pi^2}}
	\sqrt {A\left( n \right)\left( {x - a} \right)^{2n - 1}  + B\left(
		n \right)\left( {\frac{{a + b}}{2} - x} \right)^{2n - 1} }
	\cdot
	\left\|{f^{\left( {n + 1} \right)} }\right\|_{2} ,\,\,\,\,
	{\rm{if}}\,\,f^{(n+1)}
	\in L_{2}\left(\left[a,b\right]\right),\\
	\\
	\left( {b-a} \right)\left( {\frac{{n - 2}}{n}} \right)^{n - 2}
	\frac{{n^2  - 2n + 2}}{{8n\cdot (n!)^2}} \left[ {\frac{{b - a}}{4}
		+ \left| {x - \frac{{3a + b}}{4}} \right|} \right]^{n-1}\cdot
	\left( {M - m} \right),\,\,\,
	{\rm{if}}\,\, m\le f^{(n)} \le M, \\
	\\
	\left( {\frac{{n - 2}}{n}} \right)^{n - 2} \frac{{n^2  - 2n +
			2}}{{8n\cdot (n!)^2}} \left( {2^{ - n - 2}  - 2^{ - 2n - 2} }
	\right) \left( {b -  a} \right)^{n}\cdot \left\|{f^{\left( {n + 1}
			\right)} }\right\|_{\infty},\,\,\, {\rm{if}}\,\,f^{(n+1)} \in
	L_{\infty}\left(\left[a,b\right]\right),
	\end{array} \right.
	\end{multline}
	holds for all $x\in \left[a,\frac{a+b}{2}\right]$,  where
	\begin{align*}
	A\left( n \right) = \frac{{2\left( {n - 1} \right)^2 }}{{\left(
			{2n - 1} \right)\left( {2n - 2} \right)\left( {2n - 3}
			\right)}}
	\end{align*}
	and
	\begin{align*}
	B\left( n \right) = \frac{{2^{2n - 3} \left( {2n - 1}
			\right)\left( {2n - 2} \right) + 4n\left( {2n - 1} \right) + 2n^2
	}}{{\left( {2n - 1} \right)\left( {2n - 2} \right)\left( {2n - 3}
			\right)}}
	\end{align*}
	$\forall n\ge2$.
\end{theorem}

\begin{proof}
	$\bullet$  If $f^{(n+1)}\in
	L^{\infty}\left(\left[a,b\right]\right)$: Applying the first
	inequality in \eqref{eq4.4}, it is not difficult to observe that
	$\mathop {\sup}\limits_{a \le t \le b} \left\{
	{\left|{h^{\prime}_1 \left( t \right)}\right|} \right\}
	=\frac{1}{n!}\left\|{f^{\left( {n + 1} \right)}
	}\right\|_{\infty}$ and
	\begin{align*}
	\mathop {\sup}\limits_{a \le t \le b} \left\{ {\left|{h^{\prime}_2
			\left( t \right)}\right|} \right\} =\left( {\frac{{n - 2}}{n}}
	\right)^{n - 2} \frac{{n^2  - 2n + 2}}{{n}} \left[ {\frac{{b -
				a}}{4} + \left| {x - \frac{{3a + b}}{4}} \right|}
	\right]^{n-1},\qquad \forall n\ge2.
	\end{align*}
	So that
	\begin{align*}
	\left|{\mathcal{P}\left( {f;x,n} \right)} \right| \le \left( {b-a}
	\right)^2 \left( {\frac{{n - 2}}{n}} \right)^{n - 2} \frac{{n^2  -
			2n + 2}}{{12n\cdot n!}} \left[ {\frac{{b - a}}{4} + \left| {x -
			\frac{{3a + b}}{4}} \right|} \right]^{n-1}\cdot
	\frac{1}{n!}\left\|{f^{\left( {n + 1} \right)} }\right\|_{\infty}.
	\end{align*}
	
	$\bullet$  If $m\le f^{(n)}\left(t\right) \le M$, for some
	$m,M>0$: Applying the second inequality in \eqref{eq4.4}, we get
	\begin{align*}
	\left|{\mathcal{P}\left( {f;x,n} \right)} \right|\le
	\frac{n^2-2n+2}{4n\cdot n!}\left( {\frac{{n - 2}}{n}} \right)^{n -
		2} \left( {2^{ - n - 2}  - 2^{ - 2n - 2} } \right)\left( {b - a}
	\right)^{n-2} \cdot \frac{1}{n!} \left( {M -  m} \right).
	\end{align*}

	$\bullet$  If $f^{(n+1)} \in L^{2}\left(\left[a,b\right]\right)$:
	Applying the third inequality in \eqref{eq4.4}, we get
	\begin{align*}
	\left|{\mathcal{P}\left( {f;x,n} \right)} \right| \le
	\frac{{\left( {b-a} \right)}}{{n!\pi^2}} \cdot   \sqrt {A\left( n
		\right)\left( {x - a} \right)^{2n - 1}  + B\left( n \right)\left(
		{\frac{{a + b}}{2} - x} \right)^{2n - 1} }
	\cdot
	\frac{1}{n!}\left\|{f^{\left( {n + 1} \right)} }\right\|_{2}
	\end{align*}
	$\forall n\ge2$, where  $A\left( n \right)$ and $B\left( n
	\right)$ are defined
	above.\\
	
	$\bullet$  If $m\le f^{(n)}\left(t\right) \le M$, for some
	$m,M>0$: Applying the forth inequality in \eqref{eq4.4}, we get
	\begin{align*}
	\left|{\mathcal{P}\left( {f;x,n} \right)} \right| \le  \left(
	{b-a} \right)\left( {\frac{{n - 2}}{n}} \right)^{n - 2} \frac{{n^2
			- 2n + 2}}{8n \cdot n!} \left[ {\frac{{b - a}}{4} + \left| {x -
			\frac{{3a + b}}{4}} \right|} \right]^{n-1}\cdot  \frac{1}{n!}
	\left( {M - m} \right).
	\end{align*}
	By applying the forth inequality again the with dual assumptions,
	i.e.,  $f^{(n+1)} \in L^{\infty}\left(\left[a,b\right]\right)$, we
	have
	\begin{align*}
	\left|{\mathcal{P}\left( {f;x,n} \right)} \right|  &\le
	\frac{n^2-2n+2}{8n\cdot n!}\left( {\frac{{n - 2}}{n}} \right)^{n -
		2} \left( {2^{ - n - 2}  - 2^{ - 2n - 2} } \right) \left( {b - a}
	\right)^{n}\cdot \frac{1}{n!}\left\|{f^{\left( {n + 1} \right)}
	}\right\|_{\infty}.
	\end{align*}
	Hence the proof is completely established.
\end{proof}

\begin{corollary}
	\label{cor6}Let assumptions of Theorem \ref{thm14} hold. If
	moreover, $f^{\left( {n - 1} \right)}\left( {a} \right)=f^{\left(
		{n - 1} \right)}\left( {b} \right)$ $(n \ge2)$, then the
	inequality
	\begin{multline}
	\label{eq5.3}\left|{ \frac{1}{n}\left( {\frac{{f\left( x \right) +
					f\left( {a + b - x} \right)}}{2} + \sum\limits_{k = 1}^{n - 1}
			{G_k } } \right) - \frac{1}{{b - a}}\int_a^b {f\left( y \right)dy}
	}\right|
	\\
	\le\left\{
	\begin{array}{l}
	\left( {b-a} \right)^{2}\left( {\frac{{n - 2}}{n}} \right)^{n - 2}
	\frac{{n^2  - 2n + 2}}{{12n\cdot (n!)^2}} \left[ {\frac{{b -
				a}}{4} + \left| {x - \frac{{3a + b}}{4}} \right|}
	\right]^{n-1}\cdot \left\|{f^{\left( {n + 1} \right)}
	}\right\|_{\infty} ,\,\,\,\,\,\,\,\,\,{\rm{if}}\,\,f^{(n+1)} \in
	L_{\infty}\left(\left[a,b\right]\right)\\
	\\
	\left( {\frac{{n - 2}}{n}} \right)^{n - 2} \frac{{n^2  - 2n +
			2}}{{4n\cdot (n!)^2}}   \left( {2^{ - n - 2}  - 2^{ - 2n - 2} }
	\right) \left( {b -  a} \right)^{n-2}\cdot  \left( {M -  m}
	\right),\,\,\,
	{\rm{if}}\,\, m\le f^{(n)} \le M, \\
	\\
	\frac{{b-a}}{{(n!)^2\pi^2}}
	\sqrt {A\left( n \right)\left( {x - a} \right)^{2n - 1}  + B\left(
		n \right)\left( {\frac{{a + b}}{2} - x} \right)^{2n - 1} }
	\cdot
	\left\|{f^{\left( {n + 1} \right)} }\right\|_{2} ,\,\,\,\,
	{\rm{if}}\,\,f^{(n+1)}
	\in L_{2}\left(\left[a,b\right]\right),\\
	\\
	\left( {b-a} \right)\left( {\frac{{n - 2}}{n}} \right)^{n - 2}
	\frac{{n^2  - 2n + 2}}{{8n\cdot (n!)^2}} \left[ {\frac{{b - a}}{4}
		+ \left| {x - \frac{{3a + b}}{4}} \right|} \right]^{n-1}\cdot
	\left( {M - m} \right),\,\,\,
	{\rm{if}}\,\, m\le f^{(n)} \le M, \\
	\\
	\left( {\frac{{n - 2}}{n}} \right)^{n - 2} \frac{{n^2  - 2n +
			2}}{{8n\cdot (n!)^2}} \left( {2^{ - n - 2}  - 2^{ - 2n - 2} }
	\right) \left( {b -  a} \right)^{n}\cdot \left\|{f^{\left( {n + 1}
			\right)} }\right\|_{\infty},\,\,\, {\rm{if}}\,\,f^{(n+1)} \in
	L_{\infty}\left(\left[a,b\right]\right),
	\end{array} \right.
	\end{multline}
	holds for all $x\in \left[a,\frac{a+b}{2}\right]$,  where
	\begin{align*}
	A\left( n \right) = \frac{{2\left( {n - 1} \right)^2 }}{{\left(
			{2n - 1} \right)\left( {2n - 2} \right)\left( {2n - 3}
			\right)}}
	\end{align*}
	and
	\begin{align*}
	B\left( n \right) = \frac{{2^{2n - 3} \left( {2n - 1}
			\right)\left( {2n - 2} \right) + 4n\left( {2n - 1} \right) + 2n^2
	}}{{\left( {2n - 1} \right)\left( {2n - 2} \right)\left( {2n - 3}
			\right)}}
	\end{align*}
	$\forall n \ge2$.
\end{corollary}

\begin{remark}
	By setting
	$h_1\left(t\right)=\frac{1}{n!}f^{(n)}\left(t\right)k\left(t,x\right)$
	and $h_2\left(t\right)=\left(x-t\right)^{n-1}$, we obtain that
	\begin{align}
	\mathcal{C}\left( {h_1,h_2} \right) &= \frac{1}{n}\left(
	{\frac{{f\left( x \right) + f\left( {a + b - x} \right)}}{2} +
		\sum\limits_{k = 1}^{n - 1} {G_k\left(x\right) } } \right) -
	\frac{1}{{b - a}}\int_a^b {f\left( y \right)dy}
	\nonumber\\
	&\qquad-  \frac{1}{n!}\cdot\frac{{\left( {x - a} \right)^n  -
			\left( {x - b} \right)^n }}{n\left( {b - a} \right)}
	\cdot\frac{f^{\left( n \right)} \left( x \right) + f^{\left( n
			\right)} \left( {a + b - x} \right)}{2}
	\nonumber\\
	&:=\mathcal{Q}\left( {f;x,n} \right).\label{eq5.4}
	\end{align}
	Applying Theorem \ref{thm11}   Chebyshev type bounds for
	$\mathcal{Q}\left( {f;x,n} \right)$ can be proved. We shall omit
	the details.
\end{remark}

\begin{remark}
	\label{rem4}Bounds for the generalized formula \eqref{eq3.23} via
	Chebyshev-Gr\"{u}ss type inequalities can be done by setting
	$h_1\left(t\right)=\frac{(-1)^{n-1}}{n}f^{(n)}\left(t\right)$ and
	$h_2\left(t\right)=Q_{n-1}\left(t\right)S\left(t,x\right)$,
	therefore we have
	\begin{align*}
	&\mathcal{C}\left( {h_1,h_2} \right)
	\\
	&= \frac{(-1)^{n-1}}{n\left(b-a\right) } \int_a^b
	{Q_{n-1}\left(t\right) S\left( {t,x} \right)f^{\left( n \right)}
		\left( t \right)dt}
	\\
	&\qquad-   \frac{1}{b-a} \int_a^b {Q_{n-1}\left(t\right) S\left(
		{t,x} \right) dt} \times\frac{(-1)^{n-1}}{n\left(b-a\right) }
	\int_a^b { f^{\left( n \right)} \left( t \right)dt}
	\\
	&= \frac{1}{n\left(b-a\right) } \int_a^b {Q_{n-1}\left(t\right)
		S\left( {t,x} \right)f^{\left( n \right)} \left( t \right)dt}
	\\
	&\qquad-   \frac{1}{b-a} \int_a^b {Q^{\prime}_{n}\left(t\right)
		S\left( {t,x} \right) dt} \times\frac{(-1)^{n-1}}{n}\cdot
	\frac{{f^{\left( {n - 1} \right)} \left( b \right) - f^{\left( {n
					- 1} \right)} \left( a \right)}}{{b - a}}
	\end{align*}
	\begin{align*}
	&= \frac{(-1)^{n-1}}{n\left(b-a\right) } \int_a^b
	{Q_{n-1}\left(t\right) S\left( {t,x} \right)f^{\left( n \right)}
		\left( t \right)dt}
	\\
	&\qquad-   \left[ {\frac{{Q_n \left( x \right) + Q_n \left( {a + b
					- x} \right)}}{2} - \frac{{Q_{n + 1} \left( b \right) - Q_{n + 1}
				\left( a \right)}}{{b - a}}} \right]
	\\
	&\qquad\qquad\times\frac{(-1)^{n-1}}{n}\cdot \frac{{f^{\left( {n -
					1} \right)} \left( b \right) - f^{\left( {n - 1} \right)} \left( a
			\right)}}{{b - a}}
	\\
	&=\mathcal{L}\left({f,Q_n,x}\right)
	\end{align*}
	for all $x\in \left[a,\frac{a+b}{2}\right]$.    We left the
	representations to the reader.
\end{remark}

\begin{theorem}
	\label{thm15}Let $I$ be a real interval, $a,b \in I^{\circ}$
	$(a<b)$. Let $f:I \to \mathbb{R}$ be $(n+1)$-times differentiable
	on $I^{\circ}$ such that $f^{(n+1)}$ is absolutely continuous on
	$I^{\circ}$ with $\left( {\cdot - t} \right)^{n - 1} k\left(
	{t,\cdot} \right)f^{\left( n \right)} \left( t \right)$ is
	integrable. Then, for all $n\ge2$ we have
	\begin{multline}
	\left|{\mathcal{L}\left({f,Q_n,x}\right)} \right|
	\\
	\le\left\{
	\begin{array}{l}
	\frac{\left( {b - a} \right)^2}{12n}\left\| {Q_{n - 1}  + Q_{n -
			2} S\left( { \cdot ,x} \right)} \right\|_\infty \cdot
	\left\|{f^{\left( {n + 1} \right)} }\right\|_{\infty}
	,\,\,\,\,{\rm{if}}\,\,f^{(n+1)} \in
	L_{\infty}\left(\left[a,b\right]\right)\\
	\\
	\frac{1}{4n}\left( {M_1  - m_1 } \right)\left( {M_2  - m_2 }
	\right),\,\,\,\,\,\,\,\,\,\,\,\,\,\,\,\,\,\,\,\,\,\,\,\,\,\,\,\,\,\,\,\,\,\,\,\,\,\,\,
	\qquad
	{\rm{if}}\,\, m_1\le f^{(n)} \le M_1, \\
	\\
	\frac{b-a}{\pi^2n}D\left( {n,x} \right)  \cdot \left\|{f^{\left(
			{n + 1} \right)} }\right\|_{2}
	,\qquad\qquad\qquad\qquad\,\,\,\,\,\,\,  {\rm{if}}\,\,f^{(n+1)}
	\in L_{2}\left(\left[a,b\right]\right),\\
	\\
	\frac{b-a}{8n}\left\| {Q_{n - 1}  + Q_{n - 2} S\left( { \cdot ,x}
		\right)} \right\|_\infty \cdot    \left( {M_1 - m_1}
	\right),\,\,\,\,\,\,\,\,\,\,\,{\rm{if}}\,\, m_1\le f^{(n)} \le M_1, \\
	\\
	\frac{b-a}{8n}\left( {M_2  - m_2 } \right)\cdot \left\|{f^{\left(
			{n + 1} \right)}
	}\right\|_{\infty},\,\,\,\,\,\,\,\,\qquad\qquad\qquad
	{\rm{if}}\,\,f^{(n+1)} \in
	L_{\infty}\left(\left[a,b\right]\right),
	\end{array} \right.\label{eq5.5}
	\end{multline}
	holds for all $x\in \left[a,\frac{a+b}{2}\right]$,  where
	\begin{align*}
	M_2 : = \mathop {\max }\limits_{a \le t \le b} \left\{ {Q_{n - 1}
		\left( t \right)S\left( {t,x} \right)} \right\},\,\, m_2 : =
	\mathop {\min }\limits_{a \le t \le b} \left\{ {Q_{n - 1} \left( t
		\right)S\left( {t,x} \right)} \right\}
	\end{align*}
	and
	\begin{align*}
	D\left( {n,x} \right) = \left( {\int_a^b {\left| {Q_{n - 1} \left(
				t \right) + Q_{n - 2} \left( t \right)S\left( {t,x} \right)}
			\right|^2 dt} } \right)^{1/2} \qquad \forall n\ge2.
	\end{align*}
\end{theorem}

\begin{proof}
	The proof of the result follows directly by applying Theorem
	\ref{thm11} to the functions
	$h_1\left(t\right)=\frac{(-1)^{n-1}}{n}f^{(n)}\left(t\right)$ and
	$h_2\left(t\right)=P_{n-1}\left(t\right)S\left(t,x\right)$ as
	shown previously in Remark \ref{rem4} and the rest of the proof
	done using Theorem \ref{thm14}.
\end{proof}

\begin{corollary}
	\label{cor7}Let assumptions of Theorem \ref{thm15} hold. If
	moreover, $f^{\left( {n - 1} \right)}\left( {a} \right)=f^{\left(
		{n - 1} \right)}\left( {b} \right)$ $(n \ge2)$, then the
	inequality
	\begin{multline}
	\left|{\frac{1}{n}\left[ {\frac{{f\left( x \right) + f\left( {a +
						b - x} \right)}}{2} + \sum\limits_{k = 1}^{n - 1} { \left\{{ T_k
					\left(x\right)+ \widetilde{F_k } \left(a,b\right) }\right\}} }
		\right] - \frac{1}{{b - a}}\int_a^b {f\left( y \right)dy} }
	\right|
	\\
	\le\left\{
	\begin{array}{l}
	\frac{\left( {b - a} \right)^2}{12n}\left\| {Q_{n - 1}  + Q_{n -
			2} S\left( { \cdot ,x} \right)} \right\|_\infty \cdot
	\left\|{f^{\left( {n + 1} \right)} }\right\|_{\infty}
	,\,\,\,\,{\rm{if}}\,\,f^{(n+1)} \in
	L_{\infty}\left(\left[a,b\right]\right)\\
	\\
	\frac{1}{4n}\left( {M_1  - m_1 } \right)\left( {M_2  - m_2 }
	\right),
	\,\,\,\,\,\,\,\,\,\,\,\,\,\,\,\,\,\,\,\,\,\,\,\,\,\,\,\,\,\,\,\,\,\,\,\,\,\,\qquad
	{\rm{if}}\,\, m_1\le f^{(n)} \le M_1, \\
	\\
	\frac{b-a}{\pi^2n}D\left( {n,x} \right)  \cdot \left\|{f^{\left(
			{n + 1} \right)} }\right\|_{2}
	,\qquad\qquad\qquad\qquad\,\,\,\,\,\, {\rm{if}}\,\,f^{(n+1)}
	\in L_{2}\left(\left[a,b\right]\right),\\
	\\
	\frac{b-a}{8n}\left\| {Q_{n - 1}  + Q_{n - 2} S\left( { \cdot ,x}
		\right)} \right\|_\infty \cdot    \left( {M_1 - m_1}
	\right),\,\,\,\,\,\,\,\,\,\,{\rm{if}}\,\, m_1\le f^{(n)} \le M_1, \\
	\\
	\frac{b-a}{8n}\left( {M_2  - m_2 } \right)\cdot \left\|{f^{\left(
			{n + 1} \right)}
	}\right\|_{\infty},\,\,\,\,\,\,\,\qquad\qquad\qquad
	{\rm{if}}\,\,f^{(n+1)} \in
	L_{\infty}\left(\left[a,b\right]\right),
	\end{array} \right.\label{eq5.6}
	\end{multline}
	holds for all $x\in \left[a,\frac{a+b}{2}\right]$.
\end{corollary}

\begin{remark}
	In all above estimates, if one assumes that $f^{(n)}$ is convex,
	$r$-convex, quasi-convex, $s$-convex, $P$-convex, or $Q$-convex;
	we can obtain other new bounds involving convexity.
\end{remark}

\end{document}